\documentclass{amsart}

\usepackage{multirow}
\usepackage{amsfonts}
\usepackage{graphicx}
\usepackage{amscd}
\usepackage{amsmath}
\usepackage{amssymb}
\usepackage[matrix,arrow]{xy}
\usepackage[colorlinks=true, allcolors=blue]{hyperref}
\usepackage[usenames]{color}
\usepackage[mathmode, boxsize=3.7em]{ytableau}

\newtheorem{thm}{theorem}[section]
\newtheorem{theorem}[thm]{Theorem}
\newtheorem{proposition}[thm]{Proposition}
\newtheorem{lemma}[thm]{Lemma}
\newtheorem{corollary}[thm]{Corollary}
\newtheorem{remark}[thm]{Remark}

\newtheorem{definition}[thm]{Definition}

\begin{document}

\title[$\mathbb{Z}$-graded identities of the Lie algebras $U_1$]{$\mathbb{Z}$-graded identities of the Lie algebras $U_1$}
\author[Fidelis]{Claudemir Fidelis}
\thanks{C. Fidelis was supported by FAPESP grant No.~2019/12498-0}
\address{Unidade Acad\^emica de Matem\'atica, Universidade Federal de Campina Grande, Campina Grande, PB, 58429-970, Brazil
	\\ and \\Instituto de Matem\'atica e Estat\'istica da Universidade de S\~ao Paulo, SP, 05508-090, Brazil
}
\email{claudemir.fidelis@professor.ufcg.edu.br}

\author[Koshlukov]{Plamen Koshlukov}
\thanks{P. Koshlukov was partially supported by FAPESP grant No.~2018/23690-6 and by CNPq grant No.~302238/2019-0}
\address{Department of Mathematics, UNICAMP, 13083-859 Campinas, SP,  Brazil}
\email{plamen@unicamp.br}

\begin{abstract}
Let $K$ be an infinite field of characteristic different from two and let $U_1$ be the Lie algebra of the derivations of the algebra of Laurent polynomials $K[t,t^{-1}]$. The algebra $U_1$ admits a natural $\mathbb{Z}$-grading. We provide a basis for the graded identities of $U_1$ and prove that they do not admit any finite basis. Moreover, 
we provide a basis for the identities of certain graded Lie algebras with a grading such that every homogeneous component has dimension $\leq 1$, if a basis of the multilinear graded identities is known. As a consequence of this latter result
we are able to provide a basis of the graded identities of the Lie algebra $W_1$ of the derivations of the polynomial ring $K[t]$. The $\mathbb{Z}$-graded identities for $W_1$, in characteristic 0, were described in \cite{FKK}. As a consequence of our results, we give an alternative proof of the main result, Theorem 1, in \cite{FKK}, and generalize it to positive characteristic. We also describe a basis of the graded identities for the special linear Lie algebra $sl_q(K)$ with the Pauli gradings where $q$ is a prime number.
\medskip


\medskip

\end{abstract}

\maketitle

\section{Introduction}

Group graded algebras have been extensively studied in the last decades. The reader may find quite large list of references concerning recent results about gradings on algebras in the recent monograph \cite{EldKoch}. Many authors have also paid attention to grading on Lie algebras \cite{BahKoch,BahSheZai,DraEld,PatZas}. In both the associative and Lie  case, gradings that admit no proper refinement, in particular the gradings in which every component has dimension $\leq 1$, are of special interest in classification results, see \cite{BahSehZai,BahZai,DraEld,Guim,PatZas}. This development, together with recent advances in the theory of graded polynomial identities for associative algebras gave impulse to the study of graded polynomial identities for Lie algebras and superalgebras.  The asymptotics of the graded codimension growth of the Lie (Associative, Jordan) algebra $UT_n$ of $n\times n$ upper triangular matrices with an elementary grading, related to the Cartan grading on the special linear Lie algebra, was obtained in \cite{FP}. As a consequence the asymptotic behaviour of the graded codimension sequence, for every elementary grading on $UT_n$ was obtained. 

The known results on graded identities for Lie algebras are relatively few. In addition to the aforementioned paper \cite{FP}, we recall that in \cite{RZ} the Pauli grading on certain simple Lie superalgebras is defined and numerical invariants for the graded identities are studied; recall that the dimension of every non-zero homogeneous component in the grading equals 1. The Lie algebra of derivations of the polynomial ring in one variable has a canonical grading by the group $\mathbb{Z}$ such that  every non-zero component is one dimensional. The graded identities for this algebra were described in \cite{FKK}, assuming that the field is of characteristic 0. If $K$ is an algebraically closed field of characteristic different from $2$ then, up to equivalence, the fine gradings on $sl_2(K)$ are the Pauli grading, with universal group $\mathbb{Z}_2\times \mathbb{Z}_2$, and the Cartan grading with universal group $\mathbb{Z}$. Hence, up to equivalence, $sl_2({K})$ admits three non-trivial group gradings by $\mathbb{Z}_2$, $\mathbb{Z}_2\times \mathbb{Z}_2$ (Pauli grading) and $\mathbb{Z}$ (Cartan grading). Bases of the graded polynomial identities for these gradings were determined in \cite{K}. Over a field of characteristic zero, it was proved in \cite{GM2} that the algebra $sl_2(K)$ endowed with any one of these three non-trivial gradings satisfies the Specht property. In \cite{GM1} the authors prove that, in each case, the corresponding variety of graded Lie algebras is a minimal variety of exponential growth. Moreover they prove that for the Pauli grading and the Cartan grading the variety has almost polynomial growth. Recently, Fidelis, Diniz and de Souza in \cite{CDF}, over a field of characteristic zero, studied the graded identities of the special linear Lie algebra $sl_n(K)$ with the Pauli and Cartan gradings. The methods from \cite{CDF} work for the Pauli grading only in case $n$ is a prime number.

Unless otherwise stated, throughout this paper $K$ denotes an infinite field of characteristic different from two. All algebras and vector spaces which we consider are unitary and over $K$. The purpose of the present paper is to study the graded identities satisfied by the Lie algebra $U_1$ of the derivations of the algebra of Laurent polynomials $K[t,t^{-1}]$ with its natural $\mathbb{Z}$-grading. More precisely, we provide a basis for the graded identities of $U_1$.
Furthermore we prove that the graded identities of $U_1$ do not admit any finite basis. We also provide bases for the identities of certain graded Lie algebras with a grading in which every homogeneous has dimension $\leq 1$, as long as a basis of their multilinear graded identities is known. We recall that, over a field of characteristic zero, the counterpart of this result for the Lie algebra $W_1$ of the derivations of the polynomial ring $K[t]$, was studies in \cite{FKK}. We recall here that the algebras $W_n$ and $U_n$ (defined as derivations of the corresponding polynomial and Laurent polynomial rings in $n$ variables) were first studied around 1910 by E. Cartan in his classification of simple Lie algebras in characteristic 0. Later on it was discovered that these algebras, although not simple in positive characteristic, produce naturally various simple finite dimensional Lie algebras as their homomorphic images. The celebrated theorems of V. Kac \cite{kac, kacbook} classified the simple Lie algebras graded by $\mathbb{Z}$ under some natural conditions. Namely he required the algebra $L=\oplus_{i\in \mathbb{Z}} L_i$ to be of polynomial growth: $\sum_{j\le i} \dim L_j$ grows like a polynomial in $i$; $L_0$ acts irreducibly on $L_{-1}$, and $L$ is generated by degrees 0 and $\pm 1$. Later on O.  Mathieu \cite{olmat} classified the simple $\mathbb{Z}$-graded Lie algebras of polynomial growth.

The ``ordinary'' identities of $W_1$ coincide with the identities of the Lie algebra of the vector fields on the line if $K=\mathbb{R}$ is the real field. The standard Lie polynomial of order $4$ (which is of degree 5 and is alternating in 4 of its variables) is an identity of $W_1$. On the other hand, it is a long-standing open problem to determine a basis of the identities satisfied by $W_1$. The vector space of $W_n$, the derivations of the polynomial ring in $n$ variables, can be given the structure of a left-symmetric algebra, denoted by $L_n$. In \cite{koum} the authors studied the right-operator identities of $L_n$, and described a large class of general identities for $L_n$.  We hope that our results about the $\mathbb{Z}$-graded identities of $U_1$ may shed additional light on the polynomial identities satisfied by $W_1$, and consequently by $U_1$.

\section{Definitions and preliminary results} \label{Preliminaries}

We fix an infinite field $K$ of characteristic different from 2, all algebras and vector spaces we consider will be over $K$. If $A$ is an associative algebra one defines on the vector space of $A$ the Lie bracket $[a,b]=ab-ba$, and the Jordan product $a\circ b= (ab+ba)/2$, $a$, $b\in A$. Denote by $A^{(-)}$ and by $A^{(+)}$ these new structures. It is immediate that $A^{(-)}$ is a Lie algebra; the Poincar\'e--Birkhoff--Witt theorem yields that every Lie algebra is a subalgebra of some $A^{(-)}$. On the other hand, 
 $A^{(+)}$ is a Jordan algebra; the Jordan algebras of this type, as well as their subalgebras, are called special, otherwise they are exceptional.

Let $L$ be an algebra (not necessarily associative) and let $G$ be a group. A grading by $G$ (or simply a $G$-grading) on $L$ is a vector space decomposition
\begin{align}\label{gr}
\Gamma\colon L=\oplus_{g\in G}L_g,
\end{align}
such that $L_gL_h\subseteq L_{gh}$, for all $g$, $h\in G$. In this case one says that $L$ is $G$-graded. The subspaces $L_g$ are the homogeneous components of the grading and a nonzero element $a$ of $L$
is homogeneous if $a\in L_g$ for some $g\in G$; we denote this by $\|a\|_G=g$ (or simply $\|a\|=g$ when the group $G$ is clear from the context). The support of the grading is the set $\mathrm{supp}\ L=\{g\in G \mid L_g\neq 0\}$. A subalgebra (an ideal, a subspace) $B$ of $A$ is a graded subalgebra (respectively ideal, subspace) if $B=\oplus_{g\in G} A_{g}\cap B$. Let $H$ be a non-empty set, an $H$-grading $\Gamma^{\prime}: L=\oplus L^{\prime}_h$ is a refinement of (\ref{gr}) (or, equivalently, (\ref{gr}) is a coarsening of $\Gamma^{\prime}$) if for every $h\in H$ there exists $g\in G$ such that $L_h^{\prime}\subset L_g$. Important in the classification of gradings on Lie algebras are those  gradings that do not admit a proper refinement, they are called fine gradings. Below we give three examples of algebras equipped with fine gradings that will be extremely important in our paper.

Assume that $m\geq 2$ and $K$ contains $\varepsilon$ a primitive $m$-th root of the unity. Consider the matrices in $R=M_n(K)$
\[
A=\begin{pmatrix}
\varepsilon^{m-1}	& 0 & \cdots & 0 \\ 
0	& \varepsilon^{m-2} & \cdots & 0 \\ 
\vdots	& \vdots & \ddots & \vdots \\ 
0	& 0 & \cdots & 1
\end{pmatrix} \ \ \ \ \mbox{ and } \ \ \ \ B=\begin{pmatrix}
0	& 1 & 0 &\cdots & 0 \\ 
0	& 0 & 1 & \cdots & 0 \\ 
\vdots	& \vdots & \vdots & \ddots & \vdots \\ 
0	& 0 & 0 & \cdots & 1\\
1	& 0 & 0 & \cdots & 0
\end{pmatrix}.
\]
Note that $A$ and $B$ satisfy the relations  $AB=\varepsilon BA$ and $A^m=B^m=I$ where
$I$ is the identity $m\times m$ matrix. The matrices $A^iB^j$, $1\leq i$, $j \leq m$, are linearly independent over $\mathbb{K}$. Since there are $m^2$ such products, the set of these matrices is a basis for $R$. Let $L=sl_m(\mathbb{K})$ be the special linear Lie algebra, note that if $A^{i}B^{j}\neq I$ then $A^{i}B^{j}\in L$, therefore
\begin{equation}\label{pauli}
L=\oplus_{(\bar{i},\bar{j})\in\mathbb{Z}_{m}^{2}}L_{(\bar{i},\bar{j})},
\end{equation}
where $L_{(\bar{0},\bar{0})}=0$ and $L_{(\bar{i},\bar{j})}={K}(A^{i}B^{j})$ if $(\bar{i},\bar{j})\neq (\bar{0},\bar{0})$. Since 
\begin{equation}\label{comm}
[A^iB^j,A^rB^s]=(\varepsilon^{-rj}-\varepsilon^{-is})A^{i+r}B^{j+s},
\end{equation}
the decomposition (\ref{pauli}) is a grading on $L$ by the group $\mathbb{Z}_m\times \mathbb{Z}_m$. The decomposition
\[
R=\bigoplus_{(\bar{i},\bar{j})\in\mathbb{Z}_{m}^{2}}R_{(\bar{i},\bar{j})},\mbox{ with }R_{(\bar{i},\bar{j})}={K}(A^iB^j)
\]
is a grading on $R$. It is called the Pauli grading; it arises in the classification of the gradings on matrix algebras. 

Now consider the polynomial algebra $K[t]$ in one variable $t$. The derivations of the algebra $K[t]$ form a Lie algebra denoted by $W_1$, $W_1=Der(K[t])$. It is immediate that the elements $e_n=t^{n+1}d/dt$, $n\geq -1$, form a basis of $W_1$. The Lie algebra structure on the vector space $W_1$ is given by the multiplication
\begin{equation}\label{multiwitt}
[e_i, e_j] = (j-i)e_{i+j}.
\end{equation}
The algebra $W_1$ has a $\mathbb{Z}$-grading,
$W_1=\oplus_{i\in\mathbb{Z}}L_i$ where $L_i=0$ whenever $i\leq -2$, and $L_i$ is the (one-dimensional) span of $e_i$ if $i\geq -1$. The algebra $W_1$ plays an important role in the classification of simple Lie algebras. It is simple in characteristic 0, and it can be used to produce the so-called simple algebras of Cartan type in positive characteristic. It is also important in various applications.

Finally, let $A=K[t,t^{-1}]$ be the algebra of Laurent polynomials in one variable $t$. The derivations of the algebra $A$ form a Lie algebra denoted by $U_1$, $U_1=Der(A)$. As in $W_1$, it is immediate that the elements $e_n=t^{n+1}d/dt$, $n\in\mathbb{Z}$, form a basis of $U_1$ and that it has a Lie algebra structure given by the multiplication \eqref{multiwitt}. The algebra $U_1$ is $\mathbb{Z}$-graded, $U_1=\bigoplus_{i\in\mathbb{Z}}L_i$ where $L_i$ is the span of $e_i$, for each $i\in\mathbb{Z}$. We say that $U_1$ has full support on $\mathbb{Z}$, in other words $\mathrm{supp}\ U_1=\mathbb{Z}$.

There are very many other interesting and important examples of graded algebras (associative, Lie, and so on). We gave in detail the above three because we will work extensively with them in this paper. 

Let $G$ be an abelian group, let $\{X_g\}_{g\in G}$ be a family of disjoint sets $X_g = \{x_{1}^{g}, x_{2}^{g}, \ldots\}$, and let $X_G=\cup_{g\in G} X_g$. The free associative algebra $K\langle X_G \rangle$ can be given a $G$-grading as follows. We declare the monomial $x_{i_1}^{ g_1}\cdots x_{i_n}^{g_n}$ homogeneous of degree $g_1 \cdots g_n$, and then extend this by linearity. The subalgebra $L\langle X_G \rangle$ of $K\langle X_G\rangle^{(-)}$ generated by $X_G$ is the free $G$-graded Lie algebra, freely generated by the set $X_G$. Note that $L\langle X_G \rangle$ is a graded subspace of $K\langle X_G \rangle$ and that the corresponding decomposition is a $G$-grading on $L\langle X_G \rangle$. The elements of $L\langle X_G \rangle$ are called $G$-graded polynomials (or simply polynomials). The degree of a polynomial $f$ in $x_i^{g_i}$, denoted by $\deg_{x_i^{g_i}}f$, counts how many times the variable $x_i^{g_i}$ appears in the monomials of $f$, and it is defined in the ordinary way. The definitions of multilinear and multihomogeneous polynomials are the natural ones. We define the (left-normed) commutator $[l_1,\cdots, l_n]$ of $n\geq 2$ elements $l_1$, \dots, $l_n$ in a Lie algebra $L$ inductively,  $[l_1,\cdots, l_n]=[[l_1,\cdots,l_{n-1}],l_n]$ for $n>2$.

Let $L=\oplus_{g\in G}L_g$ be a Lie algebra with a $G$-grading. An admissible substitution for the polynomial $f(x_{1}^{g_1},\ldots, x_{n}^{g_n})$ in $L$ is an $n$-tuple $(a_1,\ldots, a_n)\in L^{n}$ such that $a_i\in L_{g_i}$, for $i=1$, \dots, $n$. If $f(a_1,\dots,a_n)=0$ for every admissible substitution $(a_1,\dots, a_n)$ we say that $f(x_{1}^{g_1},\dots, x_{n}^{g_n})$ is a graded identity for $L$. The set of $G$-graded polynomial identities of $L$ will be denoted by $T_G(L)$. It is a $T_G$-ideal,  that is an ideal invariant under the endomorphisms of $L\langle X_G \rangle$ as a graded algebra. The intersection of a family of $T_G$-ideals in $L\langle X_G \rangle$ is a $T_G$-ideal; given a set of polynomials $S\subseteq L\langle X_G\rangle$ we denote by $\langle S \rangle_G$ the intersection of the $T_G$-ideals of $L\langle X_G \rangle$ that contain $S$. We call $\langle S \rangle_G$ the $T_G$-ideal generated by $S$, and refer to $S$ as a basis of this $T_G$-ideal. It is well known that in characteristic 0, every $T_G$-ideal $T_G(L)$ is generated by its multilinear polynomials. Over an infinite field of positive characteristic one has to take into account the multihomogeneous polynomials instead of the multilinear ones.

The vector space $P_n$ is the subspace of $L\langle X_G \rangle$ of the multilinear polynomials in $n$ variables. Let ${\bf g}=(g_1,\dots, g_n)$ be an $n$-tuple of elements of $G$, we denote by $P_n^{\bf g}$ the subspace of $L\langle X_G \rangle$
of the multilinear polynomials of degree $n$ in the variables $x_{1}^{g_1}$, \ldots, $x_{n}^{g_n}$. Note that 
\begin{equation}\label{decomp}
\frac{P_n}{P_n\cap T_G(L)} \cong \bigoplus\limits_{\bf{g} \in G^n} \frac{P_n^{\bf g}}{P_n^{\bf g}\cap T_G(L)}.
\end{equation}

\begin{definition}\label{defgood}
	Let $L=\oplus_{g\in G} L_g$ be a grading by the group $G$ on the Lie algebra $L$. We say the $n$-tuple ${\bf g}=(g_1,\ldots,g_n)\in G^n$ is good if $P_n^{\bf g}\not \subset T_G(L)$, otherwise we say that ${\bf g}$ is bad. A good $n$-tuple ${\bf g}\in G^n$ is called standard if the commutator $[x_{1}^{g_1},\ldots, x_{n}^{g_n}]$ is not an identity for $L$. 
\end{definition}

Throughout the text, when there is no doubt about the $G$-degree of the variables, the notation of $G$-degree will be omitted. The following  definition will be used frequently. Let $L$ be a $G$-graded Lie algebra and let $I\subset T_G(L)$ be a $T_G$-ideal. 

\begin{definition}\label{sim}
	Let $f$, $g$ be two monomials in $L\langle X_G \rangle\setminus T_G(L)$. We say that $f$ and $g$ are equivalent, and denote it by $f\sim_I g$, if there exists a  scalar $\lambda$ such that $f-\lambda g$ lies in $I$.
\end{definition}

Note that the scalar $\lambda$ is different from $0$ and that $\sim_I$ is an equivalence relation. The $T_G$-ideal $I$ will be omitted from the notation if it is clear from the context.

Let $L=\oplus_{g\in G}L_g$ be a Lie algebra equipped with a $G$-grading such that $\dim L_g\leq 1$. Given $g\in G$
we denote by $C(g)$ the set $\{h\in G \mid [L_g, L_h]=0\}$. Then $L$ clearly satisfies the identities
\begin{equation}\label{basis1}
x_{1}^{g}, \quad g\notin \mathrm{supp}\ L,
\end{equation}
and
\begin{equation}\label{basis2}
[x_{1}^{g},x_{2}^{h}], \quad h\in C(g), \quad gh\in \mathrm{supp}\ L.
\end{equation}

As $\dim L_g\leq 1$ for every $g\in G$, we obtain that for every $g$, $h$, $k\in G$ such that $[x_{1}^{g},x_{2}^{h}, x_{3}^{k}]\notin T_G(L)$ there exists unique scalar $\lambda_{g,h,k}$ with the property 
\begin{equation}\label{basis3}
[x_{1}^{g},x_{3}^{k}, x_{2}^{h}]-\lambda_{g,h,k}[x_{1}^{g},x_{2}^{h}, x_{3}^{k}],
\end{equation}
is a graded identity for $L$.

Over a field of characteristic zero, a basis of the $\mathbb{Z}$-graded identities of the Lie algebra $W_1$, as well as of the special linear Lie algebra $sl_p(K)$ with the Pauli gradings where $p$ is a prime, is well-known, see \cite[Theorem 1]{FKK} and \cite[Theorem 3.13]{CDF}, respectively. The goal of the present paper is to obtain the counterpart of the previous results for the algebra $U_1$. More precisely, our main result is the following theorem.

\begin{theorem}\label{mainresult1}
	Let $K$ be a field of characteristic zero. The graded identities
	\begin{align}
	\label{Cbasis1}	[x_1^i,x_2^i] &\equiv 0\\
	\label{Cbasis3}	\alpha[x_1^a,x_2^b, x_3^c]-\beta 	[x_1^a,x_3^c, x_2^b]&\equiv 0
	\end{align}
	where $\alpha=(c-a)(b-c-a)$, $\beta=(b-a)(c-b-a)$, form a basis for the $\mathbb{Z}$-graded identities of the Lie algebra $U_1$ over $K$.
\end{theorem}

\section{$\mathbb{Z}$-Graded identities of $U_1$ in characteristic zero} \label{identities}

In this section
$K$ will be a field of characteristic zero. 
We shall prove Theorem \ref{mainresult1}. To this end
we need a series of results.

\begin{lemma}\label{3.6}
	The graded identities \eqref{Cbasis1} and \eqref{Cbasis3} hold for $U_1$.
\end{lemma}
\begin{proof}
	Let $U_1=\oplus_{i\in\mathbb{Z}} L_i$ as above. By multiplication \eqref{multiwitt} it is clear that the identities \eqref{Cbasis1} hold in $U_1$. Now we check that the identities \eqref{Cbasis3} also hold. Since  $L_i$ is the span of $e_i$ for each $i\in\mathbb{Z}$, then 	
	\[
	[e_a,e_b,e_c]=[[e_a,e_b],e_c]=(b-a)[e_{a+b},e_c]=(c-a-b)(b-a)e_{a+b+c}=\beta e_{a+b+c}.
	\]
	Similarly $[e_a,e_c,e_b]=(b-a-c)(c-a)e_{a+b+c}=\alpha e_{a+b+c}$. It follows that
	\[
	\alpha [e_a,e_b,e_c]-\beta [e_a,e_c,e_b]=0.
	\]
	The identities \eqref{Cbasis3} hold since the homogeneous components of $U_1$ have dimension $1$.
\end{proof}

In analogy with the associative case we will call \textsl{monomials} the left normed commutators in the free graded Lie algebra.

\begin{proposition}\label{monident}
Over a field of characteristic zero, every graded monomial identity of $U_1$ is consequence of the identities \eqref{Cbasis1}.
\end{proposition}
\begin{proof}
We denote by $J$ the $T_G$-ideal generated by the polynomials \eqref{Cbasis1}. As $\mathrm{char}\ K=0$ we assume that our graded monomial is multilinear. We prove the result by induction on the length $n$ of the monomial. The result is obvious for $n=1$ and for $n=2$, so we suppose $n\geq 3$. Let $M^{\prime}=[x_{1}^{a_1},\ldots, x_{n-1}^{a_{n-1}}]$. If $M^{\prime}\in T_G(U_1)$ then it lies in $J$, by the induction hypothesis, and hence $M\in J$. We assume now that $M^{\prime}\notin T_\mathbb{Z}(U_1)$. Let $g=\| M^\prime\|$ be the $G$-degree of $M'$, then the result of every admissible substitution in $M^{\prime}$ is a scalar multiple of $e_{g}$. 
Therefore $M\in T_\mathbb{Z}(U_1)$ if and only if $[x_{1}^{g},x_{2}^{a_n}]\in T_\mathbb{Z}(U_1)$. The commutator $[x_{1}^{g},x_{2}^{a_n}]$ lies in $T_\mathbb{Z}(U_1)$ if and only if $g=a_n$, and hence $M$ lies in $J$. Thus in all cases $M\in J$, as required.
\end{proof}

We denote by $I$ the $T_\mathbb{Z}$-ideal generated by the polynomials in \eqref{Cbasis1} and \eqref{Cbasis3}.  By Lemma \ref{3.6}, $I\subseteq T_{\mathbb{Z}}(U_1)$ and, as $\mathrm{char}\, K=0$, the opposite inclusion will follow from the inclusions $P_n^{\bf g}\cap T_{\mathbb{Z}}(U_1)\subset I$, for every $n$-tuple ${\bf g}$ of elements in $\mathbb{Z}$. An immediate consequence of the previous proposition is that the inclusion holds if ${\bf g}$ is a bad sequence. Therefore we look for a characterization of the bad sequences.

Let ${\bf g}$ be an $n$-tuple of elements in $\mathbb{Z}$. We denote by $\mathcal{C}({\bf g})$ the sequence formed by elements in ${\bf g}$ when we erase all entries that are equal to 0. Moreover $|\mathcal{C}({\bf g})|$ stands for the length of the new tuple, after removing the zero entries, and $|\mathcal{C}({\bf g})|_{g}$ denotes how many times the element $g\in \mathbb{Z}$ appears in
${\bf g}$, and if $g\ne 0$, also in $\mathcal{C}({\bf g})$.

\begin{lemma}\label{monosem0}
Let ${\bf g}$ be an $n$-tuple in $\mathbb{Z}^{n}$ with at least two nonzero elements. The $n$-tuple ${\bf g}$ is bad if and only if $\mathcal{C}({\bf g})$ is bad.
\end{lemma}
\begin{proof}
First it is easy to verify that if ${\bf h}$ is obtained by reordering the elements of ${\bf g}$ then there exists a graded automorphism of $L\langle X_G \rangle$ that maps $P_n^{\bf g}$ to $P_n^{\bf h}$. Hence we may assume without loss of generality that 
\[
{\bf g}=(a_1,\ldots,a_k,\underbrace{0,\ldots,0}_{n-k})
\]
where each $a_i\neq 0$. Moreover by hypothesis we have $2\le k <n$.

Let $\mathcal{C}({\bf{g}})$ be good; we can suppose it standard. Therefore the monomial
\[
[x_{1}^{a_1},x_{k+1}^{0},\ldots,x_{n}^{0},x_2^{a_{2}},\ldots,x_k^{a_{k}}]
\]
is not a graded identity of $U_1$. But this implies ${\bf g}$ is also good. Conversely, suppose that ${\bf g}$ is good. Again we can assume that ${\bf g}$ is standard. This implies that $[x^{a_1}_1,x^{a_2}_2]$ is not in $T_{\mathbb{Z}}(U_1)$. Following this reasoning and renaming the variables if necessary we can consider that $a_{k+1}=\ldots=a_{n}=0$ and
\[
[x^{a_1}_1,x^0_{k+1},\ldots,x^0_{n},x^{a_2}_2,x^{a_3}_3,\ldots,x^{a_k}_k]
\]
is not a graded identity of $U_1$. This implies that $\mathcal{C}({\bf{g}})$ is also good. 
\end{proof}

The above lemma implies that, with certain abuse of notation, we can use the same symbol ${\bf g}$ to denote the set formed by the remaining  non-zero elements of the original sequence ${\bf g}$.

The following proposition characterizes the bad sequences of $U_1$.
\begin{proposition}\label{good}
Let ${\bf g}$ be a $n$-tuple of integers such that $|\mathcal{C}({\bf g})|>0$. The sequence $\mathcal{C}({\bf g})$ is bad if and only if 
\begin{enumerate}
	\item[(i)] $|\mathcal{C}({\bf g})|\neq 1$,
	\item[(ii)] $\mathcal{C}({\bf g})$ is a sequence of elements in  $g\mathbb{Z}$, for some $g\in \mathbb{Z}$, (that is all entries in our $n$-tuple are multiples of $g$), and 
	\item[(iii)] If some multiple $-\lambda g$, $\lambda>0$, of the integer $-g$ appears in ${\bf g}$ then 
	\[
	\mathcal{C}({\bf g})=(\lambda_1(-g),\underbrace{g,\ldots,g}_{\lambda_1},\lambda_2(-g),\underbrace{g,\ldots,g}_{\lambda_2},\ldots,\lambda_l(-g),\underbrace{g,\ldots,g}_{\lambda_l}, g,\ldots, g),
	\]
up to a permutation. Here $0<\lambda_1\leq \lambda_2\leq \ldots\leq \lambda_l$. Moreover $|\mathcal{C}({\bf g})|_{g}\geq \lambda_1+\cdots+\lambda_l+2$. Otherwise, 	$\mathcal{C}({\bf g})$ is a sequence formed only by $g$.
\end{enumerate}
\end{proposition}
\begin{proof}
By using the argument of the beginning of the proof of Lemma~\ref{monosem0}, it is clear that if (i), (ii) and (iii) occur, then ${\bf g}$ is bad. 

Reciprocally, as a consequence of Lemma~\ref{monosem0}, we can consider only tuples where all elements are nonzero. This means we assume ${\bf g}=\mathcal{C}({\bf g})$. Let $j$ be the largest positive integer such that  
\[
[x_{i_1}^{a_{i_1}},\ldots, x_{i_j}^{a_{i_j}}]\notin T_\mathbb{Z}(U_1)
\]
but
\[
[x_{i_1}^{a_{i_1}},\ldots, x_{i_j}^{a_{i_j}}, x_{i_{j+1}}^{a_{i_{j+1}}}]\in T_\mathbb{Z}(U_1)
\]
where $\{x_i^{a_i}\mid i=1,\ldots,n\}$ are variables.

By using an inductive argument we can assume $j=n-1$. Suppose ${\bf g}$ is standard and $a_n=a$. We define $y=[x_{1}^{a_{1}},\ldots, x_{{n-3}}^{a_{{n-3}}}]$, $y_1=x_{{n-2}}^{a_{{n-2}}}$, $y_2=x_{{n-1}}^{a_{{n-1}}}$ and $y_3=x_{{n}}^{a}$. In this case $\|y\|+\|y_1\|+\|y_2\|=\|y_3\|$. The monomial $[y,y_1,y_3,y_2]$ is an identity for $U_1$, and hence one of the following two conditions must be met: either 1) $\|y\|+\|y_1\|=\|y_3\|$ or 2) $\|y\|+\|y_1\|+\|y_3\|=\|y_2\|$. 

Case 1) implies $\|y_2\|=0$ which contradicts the fact that all elements of our $n$-tuple ${\bf g}$ are nonzero. 

Case 2) implies $\|y\|+\|y_1\|=0$, and consequently $\|y_2\|=\|y_3\|$. Suppose that $\|y_3\|\neq \|y\|$, then
$[y,y_2,y_3,y_1]$ is an identity for $U_1$ if and only if $\|y\|+\|y_2\|+\|y_3\|=\|y_1\|$. But this implies $\|y\|=-\|y_2\|$. Therefore $\|y\|=\|y_2\|=\|y_3\|=-\|y_1\|$, or $-\|y\|=\|y_2\|=\|y_3\|=\|y_1\|$. 

If $n=4$ we are done. Hence we consider the case $n>4$. Recall that ${\bf\tilde{g}}=(a_1,\ldots,a_{n-3})$ is good, and we can assume it is standard. If $[x_{n-1}^{a_{n-1}},x_{n-2}^{a_{n-2}}]$ has degree zero in $\mathbb{Z}$ then
$[x_{n-1}^{a_{n-1}},x_{n-2}^{a_{n-2}},x_1^{a_{1}},x_2^{a_{2}},\ldots,x_{n-3}^{a_{n-3}}]$ is not an identity of $U_1$. Recall that we can assume, without loss of generality, $a_{1}\neq 0$. In this case we  repeat the argument given above. Now consider $[x_{n-1}^{a_{n-1}},x_{n-2}^{a_{n-2}}]\in T_\mathbb{Z}(U_1)$. As $n>4$, we suppose $\tilde{y}=[x_{1}^{a_{1}},\ldots, x_{{n-4}}^{a_{{n-4}}}]$, $\tilde{y}_1=x_{n-3}^{a_{n-3}}$, $\tilde{y}_2=x_{n-2}^{a}$, $\tilde{y}_3=x_{n-1}^{a}$, and $\tilde{y}_4=x_{n}^{a}$. We have to analyse two cases: 

1) If $\|\tilde{y}\|=\|\tilde{y}_4\|$, then $\|\tilde{y}_1\|=-2a$, and we consider that $[x^{a_{n-3}}_{n-3},x^{a_{n-2}}_{n-2},x^{a_{n-1}}_{n-1}]$ has degree zero in $\mathbb{Z}$. Then 
$[x^{a_{1}}_1,x^{a_{n-3}}_{n-3},x^{a_{n-2}}_{n-2},x^{a_{n-1}}_{n-1},x^{a_{2}}_2,\ldots,x^{a_{n-4}}_{n-4},x^{a}_{n}]$ is not an identity of $U_1$. As $a_1\neq 0$ we can once again repeat the argument given above. 

2) If $\|\tilde{y}\|\neq \|\tilde{y}_4\|$ then we consider  
$Y=[\tilde{y}, \tilde{y}_2,\tilde{y}_3,\tilde{y}_4,\tilde{y}_1]$. As $Y$ is also an identity for $U_1$, 
if $n=5$ then the only case to consider is $\|\tilde{y}\|+3\|\tilde{y}_4\|=\|\tilde{y}_1\|$ which implies $\|\tilde{y}_1\|=\|\tilde{y}_4\|$, and the result follows. For $n>5$, 
we consider $\|\tilde{y}\|=0$, this implies that if $\tilde{y}=[\tilde{x},x_1]$ then $\|\tilde{x}\|=-{\|x_1\|}={a}$. In this last case we have the identity 
$[\tilde{x},\tilde{y}_1, \tilde{y}_2,\tilde{y}_3,\tilde{y}_4,x_1]$ of $U_1$. 
We can assume without loss of generality that ${\|x_1\|}\neq \|\tilde{y}_4\|$. Thus we repeat this argument a finitely many times in order to obtain statement (ii). Statement (iii) follows easily repeating the above arguments.
\end{proof}

\begin{proposition}\label{ll1}
	Every graded identity for $U_1$ of degree $\leq 4$ lies in $I$.
\end{proposition}

We will separate the proof of the proposition into two lemmas.
\begin{lemma}\label{l310}
		Every polynomial identity of degree 3 for $U_1$ lies in $I$.
	\end{lemma}
	\begin{proof}
		Consider a multilinear identity of degree 3 in $U_1$
\[
f(x^a_1, x^b_2, x^c_3) = \lambda_1[x^a_1,x^b_2, x^c_3] + \lambda_2[x^a_1,x^c_3, x^b_2].
\]
If at least one of the monomials $[x^a_1,x^b_2, x^c_3]$ or $[x^a_1,x^c_3, x^b_2]$ lies in $T_{\mathbb{Z}}(U_1)$,
the result is immediate (use Lemma~\ref{monosem0}). So we suppose that neither of the monomials is an identity for $U_1$. By \eqref{Cbasis3} there exist $\alpha$, $\beta\in K$ such that
\[
\alpha[x^a_1,x^b_2, x^c_3] \equiv \beta 	[x^a_1,x^c_3, x^b_2]\pmod{I}
\]
As $[x^a_1,x^b_2, x^c_3]$ and $[x^a_1,x^c_3, x^b_2]$ are not identities of $U_1$, we have that $\alpha$ and $\beta$ are nonzero scalars. Hence
\[
[x^a_1,x^b_2, x^c_3] \equiv \alpha^{-1}\beta 	[x^a_1,x^c_3, x^b_2]\pmod{I}
\]
This implies that
\[
f\equiv (\lambda_1\alpha^{-1}\beta+\lambda_2)[x^a_1,x^c_3, x^b_2]\pmod{I}
\]
and Lemma~\ref{monosem0} implies the result.
	\end{proof}

The identity
\begin{equation}\label{jacobigen}
[x_1, x_2, x_3, x_4] + [x_2, x_1, x_4, x_3] + [x_4, x_3, x_2, x_1] + [x_3, x_4, x_1, x_2] = 0
\end{equation}
holds in every Lie algebra (just write it as $[[x_1,x_2],[x_3,x_4]] + [[x_3,x_4],[x_1,x_2]]=0$). In particular, the previous identity also holds for every graded Lie algebra.

Let $n$ be a positive integer. Given ${\bf g}\in \mathbb{Z}^n$ and a subset $\mathfrak{S}=\{s_1,\dots,s_{n}\}\subset \mathbb{N}$, with $s_1<\ldots<s_{n}$, denote by $P_{\mathfrak{S}}^{\bf g}$ the subspace of $L\langle X_\mathbb{Z} \rangle$ of the multilinear polynomials in the variables $\{x_{s_1}^{ a_1},\ldots, x_{s_{n}}^{a_n}\}$. The elements of $\mathbb{Z}$ and the $n$-tuple ${\bf g}$ may be omitted from the notation if no ambiguity arises, thus we write $P_{\mathfrak{S}}$ and $x_i$ instead of $P_\mathfrak{S}^{\bf g}$ and $x_{i}^{a_i}$, respectively. Let $\sigma$ be a permutation in $S_{n-1}$. Denote by $N_{\sigma}$ the monomial $[x_{s_n},x_{s_{\sigma(1)}},\cdots,x_{s_{\sigma(n-1)}}]$. The set $\{N_\sigma \mid \sigma \in S_{n-1}\}$ is a basis for the vector space $P_\mathfrak{S}$.

We use the notation $\sim$ for the equivalence in Definition \ref{sim} relative to $I$. We indicate above the symbol $\sim$ the identity or result from which the equivalence follows, thus for example $M\stackrel{L.\, \ref{l310}}{\sim} N$ means that the equivalence $M\sim N$ follows from Lemma \ref{l310}.
 
\begin{lemma}\label{degree4}
Let ${\bf a}=(a_1,a_2,a_3,a_4)$ be a good tuple. If $N_{\sigma}, N_{\tau} \notin T_\mathbb{Z}(U_1)$ for some $\sigma, \tau\in S_3$ then $N_\sigma	\sim N_{\tau}$. In particular, every graded identity of degree 4 for $U_1$ lies in $I$.
\end{lemma}
\begin{proof}
	We suppose, without loss of generality, that $[x_4,x_{1},x_{2},x_{3}]$ is not an identity for $U_1$ and $a_1+a_2+a_3\neq a_4$. It is sufficient to prove that for every $\sigma \in S_3$ such that $N_\sigma\notin T_\mathbb{Z}(U_1)$, we have
\[
N_\sigma=[x_4,x_{\sigma(1)},x_{\sigma(2)},x_{\sigma(3)}]\sim [x_4,x_{1},x_{2},x_{3}].
\]
We shall analyse each one of the following cases. 

		(i) If $[x_4,x_{2},x_{1},x_{3}]$ is not an identity for $U_1$ then we apply the identity \eqref{Cbasis3}.
		
		(ii) The case of $[x_4,x_{1},x_{3},x_{2}]$ is dealt with in the same manner as that of $[x_4,x_{2},x_{1},x_{3}]$.
		
		(iii) Suppose the monomial $[x_4, x_3, x_1, x_2]$ is not a graded identity for $U_1$. 

Assume $[x_4,x_1,x_3,x_2]$ is not graded identity for $V$, then the result follows applying the item (ii). Therefore, we can assume without loss of generality that
the polynomial $[x_4,x_1,x_3,x_2]$ is a graded identity for $V$. In this case,  we must have either $\|x_4\|=\|x_1\|$, or $\|x_3\|=\|x_4\|+\|x_1\|$, or else $\|x_4\|+\|x_1\|+\|x_3\|=\|x_2\|$. The first and last of these cannot happen since $N_\sigma$ and $[x_4, x_3, x_1, x_2]$ are not graded identities for $V$. Hence, $\|x_3\|=\|x_4\|+\|x_1\|$. 

Now, consider	$[x_1, x_2, [x_3, x_4]]$ (and consequently $[x_4, x_3, [x_1, x_2]]$) is a graded identity for $U_1$. This implies  $[x_4, x_3, x_2, x_1]$ cannot be a graded identity of $U_1$. Therefore 
		\begin{equation}\label{1eq}
		[x_4, x_3, x_2, x_1]\stackrel{Eq.\ \eqref{Cbasis3}}{\sim} [x_4, x_3, x_1, x_2].
		\end{equation} 
Since $[x_1, x_2, [x_3, x_4]]$ is a graded identity we must have either $\|x_1\|=\|x_2\|$, or $\|x_3\|=\|x_4\|$, $\|x_3\|=\|x_4\|$, or $\|x_1\|=\|x_2\|$, or else $\|x_1\|+\|x_2\|=\|x_3\|+\|x_4\|$. 

The first  cannot happen since  $[x_4, x_3, x_1, x_2]$ is not an identities for $V$. Notice that if $\|x_1\|=\|x_2\|$, we then have 
	$$[x_4,x_{1},x_{2},x_{3}]\sim_J [x_4,x_{2},x_{1},x_{3}]\sim_J[x_4,x_{2},[x_{1},x_{3}]].$$
	On the other hand,
	$$[x_4, x_3, x_1, x_2]\sim_J [x_4, [x_3, x_1], x_2].$$
	And the result follows by identity \eqref{basis2}. Thus, the last item is the only case to consider.

The first two of these cannot happen since ${\bf g}$ is standard and $[x_4, x_3, x_1, x_2]$ is not an identity for $U_1$. Moreover $\|x_3\|=\|x_1\|$ if and only if $\|x_4\|=\|x_2\|$, and we have
\[
[x_1, x_2, x_3, x_4]\equiv_I[x_1, [x_2, x_3], x_4]=[x_4, [x_2, x_3], x_1]+[x_1, x_4,[x_2,x_3]].
\]
But $\|x_3\|=\|x_1\|$ and $\|x_4\|=\|x_2\|$ imply that $[x_1, x_4,[x_2,x_3]]$ lies in $T_\mathbb{Z}(U_1)$, and therefore
\[
[x_1, x_2, x_3, x_4]\equiv_I [x_4, [x_2, x_3], x_1]\stackrel{Eq.\ \eqref{1eq}}{\sim}[x_4, x_3, x_1, x_2].
\]
On the other hand
\[
[x_4, x_1, x_2, x_3]\equiv_I [x_4, [x_1, x_2], x_3]=-[x_1, x_2,x_4, x_3]\stackrel{Eq.\ \eqref{1eq}}{\sim}[x_1, x_2, x_3, x_4].
\]

Therefore we can suppose that the $G$-degree of the variables are pairwise distinct. As $\|x_3\|=\|x_1\|+\|x_4\|$, we then have
\[
[x_4,x_1,x_2,x_3]\equiv_I[x_4,x_1,[x_2,x_3]]\equiv_I[x_4,[x_1,[x_2,x_3]]]+[x_4,[x_2,x_3],x_1].
\]

Therefore 
\[
[x_4,x_1,x_2,x_3]\equiv_I[x_2,x_3,x_1,x_4]+[x_4,[x_2,x_3],x_1].
\]
If $[x_4,[x_2,x_3],x_1]$ is not an identity of $U_1$, we have the result. Otherwise we have $\|x_3\|=\|x_1\|+\|x_4\|$, $\|x_4\|=\|x_2\|+\|x_3\|$,
\begin{equation}\label{monbom}
[x_4,x_1,x_2,x_3]\equiv_I[x_2,x_3,x_1,x_4]\equiv_I[x_1,x_3,x_2,x_4]-[x_1,x_2,x_3,x_4], 
\end{equation}		
		and
\[
[x_4,x_3,x_1,x_2]= [x_1,[x_3,x_4],x_2]\equiv_I[x_1,x_3,x_4,x_2].
\]
		If $[x_1,x_3,x_2,x_4]$ is not an identity we are done. Hence, in the list above, we can add $\|x_1\|+\|x_3\|=\|x_2\|$, and 
		\begin{equation}\label{monbom1}
[x_4,x_1,x_2,x_3]\equiv_I-[x_1,x_2,x_3,x_4].
		\end{equation}
	By \eqref{monbom}, the monomial $[x_3,x_2,x_1,x_4]$ is not an identity of $U_1$, and then the result follows by the above equivalence, since 
\[
[x_4,x_3,x_2,x_1]\stackrel{Eq.\ \eqref{monbom1}}{\sim}[x_3,x_2,x_1,x_4].
\]
We are done.
		
Now we suppose that $[x_1, x_2, [x_3, x_4]]$ (and consequently $[x_4, x_3, [x_1, x_2]]$) is not a graded identity for $U_1$. This clearly follows from Eq.~\eqref{jacobigen}. 
		
		(iv) The case $[x_4, x_3, x_2, x_1]$ is analogous to the above case (iii).
		
		(v) Finally suppose $[x_4,x_2,x_3,x_1]$ is not an identity of $U_1$. If both monomials $[x_4,x_3,x_2,x_1]$ and $[x_4,x_2,x_1,x_3]$ are not identities, then the result follows as above.  Hence we assume the monomial $[x_4,x_3,x_2,x_1]$ is an identity of $U_1$. It follows that either $\|x_4\|=\|x_3\|$, or $\|x_4\|+\|x_3\|=\|x_2\|$, or else	$\|x_4\|+\|x_3\|+\|x_2\|=\|x_1\|$. The latter case is impossible because $[x_4,x_2,x_3,x_1]$ is not an identity of $U_1$. Suppose that $\|x_4\|=\|x_3\|$, then
\begin{equation}\label{xx1}
[x_4,x_2,x_3,x_1]=-[x_2,x_4,x_3,x_1]\equiv_I-[x_2,x_3,x_4,x_1].
\end{equation}		
The identity $[x_4,x_2,x_1,x_3]$ implies $\|x_4\|+\|x_2\|=\|x_1\|$, that is $\|x_3\|+\|x_2\|=\|x_1\|$. Therefore 
	\[
	[x_4,x_2,x_3,x_1]\stackrel{Eq.\ \eqref{xx1}}{\sim}[x_2,x_3,x_4,x_1]\equiv_I[x_1,x_4,[x_2,x_3]]\stackrel{Eq.\ \eqref{Cbasis3}}{\sim}[x_4,x_1,x_2,x_3].
	\]
Now suppose $\|x_4\|+\|x_3\|=\|x_2\|$ and  $\|x_4\|+\|x_2\|=\|x_1\|$. This implies that
			\begin{equation}\label{gg1}
			[x_4,x_1,x_2,x_3]\sim[x_1,x_2,x_4,x_3]
		\end{equation}
		and 
\[
[x_4,x_2,x_3,x_1]\sim [x_4,x_2,[x_3,x_1]]\sim [x_1,x_3,[x_4,x_2]]=[x_1,x_3,x_4,x_2]-[x_1,x_3,x_2,x_4],
\]
that is
				\begin{equation}\label{gg2}
		[x_4,x_2,x_3,x_1]\sim[x_1,x_3,x_4,x_2]-[x_1,x_3,x_2,x_4].
		\end{equation}
		If $[x_1,x_3,x_2,x_4]$ is not an identity of $U_1$, then the result follows since
\[
[x_4,x_2,x_3,x_1]\stackrel{Eq.\ \eqref{gg2}}{\sim} [x_1,x_3,x_2,x_4]\stackrel{Eq.\ \eqref{Cbasis3}}{\sim} [x_1,x_2,x_3,x_4]\stackrel{Eq.\ \eqref{gg1}}{\sim}[x_4,x_1,x_2,x_3].
\]
If $[x_1,x_3,x_2,x_4]$ is an identity then either $\|x_2\|=\|x_1\|+\|x_3\|$ or  $\|x_1\|=\|x_3\|$, in addition to $\|x_4\|+\|x_2\|=\|x_1\|$. If $\|x_2\|=\|x_1\|+\|x_3\|$ holds then
$\|x_1\|+\|x_3\|=\|x_4\|+\|x_3\|$ and thus $\|x_1\|=\|x_4\|$, that is
the graded monomial $[x_4,x_1,x_2,x_3]$ lies in $T_G(U_1)$ which is an absurd. On the other hand, when $\|x_1\|=\|x_3\|$ we also reach an absurd, since $\|x_2\|=\|x_4\|+\|x_3\|=\|x_4\|+\|x_1\|=2\|x_4\|+\|x_2\|$ and then $\|x_4\|=0$. The latter equality implies $\|x_1\|=\|x_2\|$ which contradicts $[x_4,x_2,x_3,x_1]$ to be not graded identity for $V$, since $\|x_1\|=\|x_2\|=\|x_3\|$ and $\|x_4\|=0$. The proof is complete now.
\end{proof}

\begin{lemma}\label{incl}
	Let ${\bf g}\in \mathbb{Z}^n$ be a standard $n$-tuple and let $\sigma\in S_{n-1}$. If $N_{\sigma}\sim [x_{1},\ldots, x_{n}]$, whenever $N_{\sigma}\notin T_\mathbb{Z}(U_1)$,
	then $P_n^{\bf g}\cap T_\mathbb{Z}(U_1)\subseteq I$.
\end{lemma}
\begin{proof}
	Note that if $\sigma$, $\tau\in S_{n-1}$ are permutations such that $N_{\sigma}$, $N_{\tau}\notin T_\mathbb{Z}(L)$ then $N_{\sigma}\sim N_{\tau}$. Denote $A=\{\sigma \in S_{n-1} \mid N_{\sigma}\in T_\mathbb{Z}(U_1)\}$ and $B=S_{n-1}\setminus A$. Let 
	\[
	f=\sum_{\sigma \in S_{n-1}} \lambda_\sigma N_{\sigma},
	\]
where $\lambda_{\sigma}\in {K}$, be a polynomial in $P_n^{\bf g}\cap T_\mathbb{Z}(U_1)$. Lemma~\ref{monident} implies that 
\[
f\equiv_I \sum_{\sigma \in B} \lambda_\sigma N_{\sigma}.
\]
Let $\tau$ be a permutation in $B$, then $N_{\sigma}\sim N_{\tau}$ for every $\sigma \in B$. Therefore there exists $\lambda \in {K}$ such that $f\equiv_I \lambda N_{\tau}$. Since $f\in T_\mathbb{Z}(U_1)$ and $N_{\tau}\notin T_\mathbb{Z}(U_1)$ it follows that $\lambda=0$, hence $f\in I$. 
\end{proof}

\begin{definition}\label{redu}
	A monomial $M\notin T_\mathbb{Z}(U_1)$ is reducible in the variables $x_i$ and $x_j$ if there exist ${\bf h}\in G^{n-1}$, $\mathfrak{S}\subset \{1,\dots, n\}$ with $|\mathfrak{S}|=n-1$ elements, and $N$ in the  canonical basis of $P_{\mathfrak{S}}^{\bf h}$ such that $M$ is equivalent to the monomial obtained from $N$ by an admissible substitution of $x_{s_k}^{h_k}$ with $[x_i,x_j]$, for some $k<n-1$.
\end{definition}

\begin{lemma}\label{l1}
	Let ${\bf g}\in \mathbb{Z}^n$ be a standard $n$-tuple. Let $\sigma\in S_{n-1}$ be such that if $N_{\sigma}$ is reducible in $x_i$, $x_j$ then $i=n$ or $j=n$. Then $P_n^{\bf g}\cap T_\mathbb{Z}(U_1)\subset I$.
\end{lemma}
\begin{proof}
	Let $\Sigma=\{\tau\in S_{n-1}\mid N_{\tau}\sim N_{\sigma}\}$. Given $\sigma^{\prime}\in \Sigma$ and $k\leq n-2$ we have 
	\[
	N_{\sigma^{\prime}}-N_{\sigma^{\prime}\circ (k,k+1)}=[x_n,\cdots, [x_{\sigma^{\prime}(k)},x_{\sigma^{\prime}(k+1)}],\ldots, x_{\sigma^{\prime}(n-1)}].
	\] 
If this monomial does not lie in $T_\mathbb{Z}(U_1)$ then it is equivalent to $N_{\sigma^{\prime}}$. It follows that $N_{\sigma^{\prime}}$ is reducible in $x_{\sigma^{\prime}(k)}$, $x_{\sigma^{\prime}(k+1)}$, which is a contradiction. Therefore $N_{\sigma^{\prime}}\sim N_{\sigma^{\prime}\circ (k,k+1)}$. This equivalence implies $\sigma^{\prime}\circ (k,k+1)\in \Sigma$. Since the transpositions $(k,k+1)$, $k\leq n-2$ generate $S_{n-1}$, we conclude $\Sigma=S_{n-1}$. Therefore Lemma~\ref{incl} implies  $P_n^{\bf g}\cap T_\mathbb{Z}(U_1)\subset I$.
\end{proof}

\begin{lemma}\label{l2}
	Let ${\bf g}\in \mathbb{Z}^n$ be a standard $n$-tuple. Let $\sigma$ be a permutation in $S_{n-1}$ such that $N_{\sigma}\notin T_\mathbb{Z}(U_1)$. If $N_{\sigma}$ is reducible for only one pair $x_i$, $x_j$ of variables with $i$, $j\neq n$ then $P_n^{\bf g}\cap T_\mathbb{Z}(U_1)\subset I$.
\end{lemma}
\begin{proof}
Let $\Sigma=\{\tau\in S_{n-1}\mid N_{\tau}\sim N_{\sigma}\}$ and $\mathcal{N}=\{N_{\sigma^{\prime}}\mid \sigma^{\prime}\in\Sigma\}$. Since $N_{\sigma}$ is reducible for exactly one pair  $x_i$, $x_j$ of variables with $i$, $j\neq n$, an argument analogous to that in the proof of the previous lemma yields that if we permute  any two adjacent variables in a monomial in $\mathcal{N}$, except for the pair $x_i$, $x_j$, the resulting monomial also lies in $\mathcal{N}$. We assume without loss of generality that $x_i$ appears to the left of $x_j$ in $N_{\sigma}$. Clearly there exists a monomial $N_{\sigma^{\prime}}$ in $\mathcal{N}$ such that  $x_i$ is adjacent and to the left of $x_j$. If the monomial obtained from $N_{\sigma^{\prime}}$ by permuting $x_i$ and $x_j$, that is the monomial $N_{(\sigma^{\prime}\circ (k,k+1))}$, for an adequate $k$, is not an identity then $N_{\sigma^{\prime}}\sim N_{(\sigma^{\prime}\circ (k,k+1))}$. We conclude that $\Sigma=S_{n-1}$ and the statement follows from Lemma \ref{incl}. Otherwise whenever $x_i$ appears adjacent and to the right of $x_j$ the resulting monomial is an identity, and whenever it appears to the left the resulting monomial is equivalent to $N_{\sigma}$. For every $k\notin\{n,i,j\}$, we conclude that the monomials are $[x_n,x_i,[x_j,x_k]]$, $[x_n,x_j,x_i,x_k]$, $[x_n,[x_i,x_k],x_j]$, $[x_n,x_k,x_j,x_i]$. The result now follows from Proposition \ref{ll1}.
\end{proof}

	If the indices $i$, $j$, in Definition~\ref{redu}, are different from $n$ then $s_{n-1}=n$. Moreover, the degrees of the remaining $n-2$ variables of the monomial $N$ are given by the sequence in the following definition.

\begin{definition}
	Let $n$ be a positive integer, ${\bf g}=(a_1,\ldots,a_n)\in \mathbb{Z}^n$, and $1\leq i,j \leq n-1$. We denote by $\Delta_{i,j}({\bf g})=(h_1,\dots, h_{n-2})$ the sequence where $(h_1,\dots, h_{n-3})$ is the sequence obtained from $(a_1,\ldots,a_{n-1})$ deleting $a_i$ and $a_j$, $h_{n-2}=a_{i}+a_{j}$.
\end{definition}

\begin{lemma}\label{l3}
	Let $n\geq 5$ and let ${\bf g}=(a_1,\ldots,a_{n})\in \mathbb{Z}^n$ be a standard $n$-tuple. If there exists a permutation $\sigma \in S_{n-1}$ such that $N_{\sigma}$ is reducible for the distinct pairs $x_i$, $x_j$, and $x_k$, $x_l$ with $i$, $j$, $k$, $l\neq n$ then at least one of the tuples $\Delta_{i,j}({\bf g})$ or $\Delta_{k,l}({\bf g})$ is good.
\end{lemma}
\begin{proof}
Assume that both $\Delta_{i,j}({\bf g})$ and $\Delta_{k,l}({\bf g})$ are bad. Using the same argument as in the proof of Lemma \ref{monosem0}, we can suppose that $a_i$, $a_j$, $a_k$ and $a_l$ are nonzero in $\mathbb{Z}$, since $(a_1,\dots, a_{n-1})$ is good. If $\{i,j\}\cap \{k,l\}=\emptyset$ then, since $\Delta_{i,j}({\bf g})$ is bad, Proposition \ref{good} and $n\geq 5$, imply that there exists $g\in \mathbb{Z}$ such that one of the following holds:
\begin{enumerate}
	\item[1)] $a_i=-a_j=g$ (or $a_k=-a_l=g$);
	\item[2)] $a_i=\lambda_r(-g)$, $a_j=\lambda_s(-g)$ and $a_k=\lambda_t(-g)$, $a_l=\lambda_v(-g)$; 
	\item[3)] $a_i=\lambda_r(-g)$, $a_j=g$ and $a_k=\lambda_s(-g)$, $a_l=g$;
	\item[4)] $a_i=\lambda_r(-g)$, $a_j=g$ and $a_k=\lambda_t(-g)$, $a_l=\lambda_v(-g)$; or
 \item[5)] $a_i=\lambda_r(-g)$, $a_j=\lambda_s(-g)$ and $a_k=\lambda_t(-g)$, $a_l=g$, 
\end{enumerate} 
for some $r$, $s$, $t$ and $v$, because $N_\sigma$ is reducible in $x_i$, $x_j$, and in $x_k$, $x_l$ with $i$, $j$, $k$, $l\neq n$.

But in cases 1, and 2, the sequence $(a_1,\dots, a_{n-1})$ is bad. Case 3 leads to the same conclusion because we increase the degree of $-g$ and this neutralizes with increasing the degree of variables of degree $g$. In the remaining two cases the argument is similar. Hence in each case we conclude that $(a_1,\ldots,a_{n-1})$ is a bad sequence, which is a 
contradiction because $M=[x_1,\ldots, x_n]\notin T_G(L)$.

Assume otherwise $\{i,j\}\cap \{k,l\}\neq\emptyset$, and without loss of generality that $j=k$. As $n>4$, we analyse the following five cases:
\begin{enumerate}
	\item[1$^\prime$)] $a_i=-a_j=g$ (or $a_l=-a_j=g$);
	\item[2$^\prime$)] $a_i=-a_j=-g$ (or $a_l=-a_j=-g$);
	\item[3$^\prime$)] $a_i=\lambda_r(-g)$, $a_j=\lambda_s(-g)$ and $a_k=\lambda_t(-g)$; 
	\item[4$^\prime$)] $a_i=\lambda_r(-g)$, $a_j=g$ and $a_k=\lambda_s(-g)$; 
	\item[5$^\prime$)] $a_i=a_k=g$ and $a_j=\lambda_s(-g)$.
\end{enumerate} 
An argument similar to that of $\{i,j\}\cap \{k,l\}=\emptyset$ convinces us that in each of these cases the sequence $(a_1,\ldots, a_{n-1})$ would
be bad, a contradiction. 
\end{proof}

Now we have all the ingredients for the proof of the main result in the section.

\begin{proof}[Proof of Theorem \ref{mainresult1}]
We prove, by induction on $n$, that $P_n^{\bf g}\cap T_\mathbb{Z}(U_1)\subseteq I$ for every $n$ and every ${\bf g}\in \mathbb{Z}^n$. 

If ${\bf g}$ is a bad sequence the inclusion follows from Proposition \ref{monident}.  Now assume that ${\bf g}$ is good. If ${\bf h}$ is obtained by reordering the elements of ${\bf g}$, then there exists an automorphism of $L\langle X_{\mathbb{Z}} \rangle$ that maps $P_n^{\bf g}$ to $P_n^{\bf h}$, hence we may assume without loss of generality that ${\bf g}$ is standard. It follows from Lemmas \ref{l1} and \ref{l2} that we can restrict to the case where for every permutation $\sigma \in S_{n-1}$ such that $N_{\sigma}\notin T_\mathbb{Z}(U_1)$, the monomial $N_{\sigma}$ is reducible for distinct pairs $(i,j)$ and $(k,l)$ of elements different from $n$. Lemma \ref{l3} implies that we may assume $\Delta_{i,j}({\bf g})=(h_1,\ldots,h_{n-2})$ is good.

We claim that there exists a multilinear monomial $N_1$ such that $N_{\sigma}\sim [N_1, x_n]$. The induction hypothesis implies that $[N_1, x_n]\sim M$. Therefore $N_{\sigma}\sim M$ where $M=[x_1,\dots, x_n]$. This implies that the inclusion $P_n^{\bf g}\cap T_{\mathbb{Z}}(U_1)\subseteq I$ follows from Lemma \ref{incl}.

Note that ${\bf h}=\Delta_{i,j}({\bf g})$ encodes, up to a permutation, the degrees of the variables in $N$ in Definition \ref{redu} apart from $x_n$. Let $S$ be the set of variables in $N$ and let $S^{\prime}=S\setminus\{x_n\}$. Since $\Delta_{i,j}({\bf g})$ is good there exists a monomial $N_0\in P_{S^{\prime}}\setminus T_{\mathbb{Z}}(U_1)$. Since ${\bf g}=(a_1,\ldots,a_n)$ is standard, $[x_1,\ldots, x_n]\notin T_{\mathbb{Z}}(U_1)$, hence $[L_g, L_{a_n}]\neq 0$ where $g=a_1+\cdots+a_{n-1}$ is the $\mathbb{Z}$-degree of $N_0$. This implies $[N_0,x_n]\notin T_{\mathbb{Z}}(U_1)$. The monomial $N$ has degree $n-1$ and is not an identity for $L$ because $N_{\sigma}\notin T_{\mathbb{Z}}(U_1)$, hence the induction hypothesis implies $[N_0,x_n]\sim N$. The monomial $N_{\sigma}$ is equivalent to the one obtained from $N$ by substituting a variable by $[x_i,x_j]$. Let $N_1$ denote the monomial obtained from $N_0$ by the same substitution. We then have $[N_1,x_n]\sim N_{\sigma}$. 
\end{proof}

The results in this section also allow us to give an alternative proof for \cite[Theorem 1]{FKK}. We observe that in condition (iii) of the Proposition \ref{good} for the Lie algebra $W_1$ with its natural $\mathbb{Z}$-grading one will have $g = 1$, and in this case each $\lambda_i=1$.

\section{$\mathbb{Z}$-Graded identities of $U_1$ over an infinite field of characteristic $p>2$} 

In this section $K$ denotes an infinite field of characteristic $p>2$. We deduce the following analogue of Theorem \ref{mainresult1}. 

\begin{theorem}\label{mainresult2}
	Let $K$ be an infinite field of characteristic $p>2$. The graded identities
\begin{equation}\label{C1identp}
[x_1^a,x_2^b] \equiv 0
\end{equation}	
	where $(b-a)$ is a multiple of $p$, and
	\begin{equation}\label{C2identp}
	\alpha[x_1^a,x_2^b, x_3^c]-\beta 	[x_1^a,x_3^c, x_2^b]\equiv 0
	\end{equation}	
	where $\alpha=(c-a)(b-c-a)$, $\beta=(b-a)(c-b-a)$, form a basis for the $\mathbb{Z}$-graded identities of the Lie algebra $U_1$ over $K$.
\end{theorem}

\begin{remark}
	If $K$ is a field of characteristic zero 
	the identities \eqref{C1identp} are substituted by $[x_1^a, x_2^a]=0$, for every $a\in\mathbb{Z}$.
\end{remark}

Let $L=\oplus_{g\in G}L_g$ be a Lie algebra  with a $G$-grading such that $\dim L_g\leq 1$. The next theorem will provide us with conditions to determine a basis for the graded identities of $L$ in case its multilinear graded identities are known.

We consider a monomial $M = [x_{i_1},\ldots, x_{i_n}]$ of length $n$ in $L\langle X_G\rangle$, and denote by $h(M)$ the sequence $(h_1,\ldots, h_n)$ where $h_k$ is the degree of the variable $x_{i_k}$.

\begin{lemma}\label{Diogo3.7.}
	Let $M_1$, $M_2$ be monomials such that $h(M_1) = h(M_2)$, then $M_1 \in T_{\mathbb{Z}}(U_1)$ if and only if $M_2 \in T_{\mathbb{Z}}(U_1)$.
\end{lemma}
\begin{proof}
	The statement follows from $\dim L_g\le 1$. If $g\in \mathrm{supp}\ U_1$ then $L_g$ is span of an element $d_g$. Therefore for every admissible substitution $(a_1,\ldots,a_r)$ of elements in $U_1$, there exist $\lambda_i\in K$ such that $a_i=\lambda_i d_{g_i}$. Here we assume, without loss of generality, that each $g_i\in \mathrm{supp}\ L$. We conclude that
\[
M_1(a_1,\ldots,a_r)=[a_{i_1},\ldots, a_{i_n}]=	\lambda_{i_1}\cdots \lambda_{i_n}[d_{g_{i_1}},\ldots, d_{g_{i_n}}]=\lambda_{i_1}\cdots \lambda_{i_n}M_1(d_{g_1},\ldots,d_{g_r}).
\]
	As $K$ is an infinite field we have $M_1\in T_{\mathbb{Z}}(U_1)$ if and only if $M_1(d_{g_1},\ldots,d_{g_r})=0$.
	Similarly for $M_2$, and the result follows.
\end{proof}

The above result is more general in the following sense.
\begin{theorem}\label{identimult}
	Let $K$ be an infinite field of characteristic different from two and let $L=\oplus_{g\in G}L_g$ be a Lie algebra  with a $G$-grading such that $\dim L_g\leq 1$. Then every graded identity of $L$ follows from its multilinear graded identities.
\end{theorem}
\begin{proof}
	Since $K$ is infinite, we can consider the multihomogeneous graded identities. Let $f=f(x_{1}^{g_1},\ldots, x_{r}^{g_r})$ be a multihomogeneous graded identity of $L$. As $\dim_KL_g\leq 1$  for every $g\in G$, we have that each admissible substitution $\varphi$ can be taken to map $x_{i}^{g_i}$ to $\xi_id_{g_i}$ where
	 $\xi$'s are commutative and associative variables over $K$. Here, as above, the $d_{g_i}$ span the vector spaces  $L_{g_i}$. Hence
	  we can assume that every $g_i\in \mathrm{supp}\ L$, and we have
	\begin{equation}\label{idemult}
	\varphi(f(x_{1}^{g_1},\ldots,x_{r}^{g_r}))=\xi_1^{n_1}\cdots \xi_r^{n_r}f(d_{g_1},\ldots,d_{g_r}).
	\end{equation}
	Fix a $g_i$ with $\deg_{x_i^{g_i}}f=n_i\geq 1$. Take new variables $x_{i,j}^{g_i}$ where $1\leq j\leq n_i$, and define the multilinear graded polynomial $h(x_{1,1}^{g_1},\ldots,x_{1,n_1}^{g_1},x_{2,1}^{g_2},\ldots, x_{r,n_r}^{g_r})$
	such that
\[
h(\underbrace{x_{1}^{g_1},\ldots,x_{1}^{g_1}}_{n_1},x_{2}^{g_2},\ldots, x_{r}^{g_r})=f(x_{1}^{g_1},\ldots,x_{r}^{g_r}).
\]
But $x_{i,j}^{g_i}$ can be evaluated to $d_{g_i}$ for each $1\leq i\leq r$ and $1\leq j\leq n_i$. Then Eq.~\eqref{idemult} implies that $h$ is a graded identity and that $f$ is a consequence of $h$.
\end{proof}

\begin{remark}
The choice of the polynomial $h$ in the above theorem
need not be given by the multilinearisation process. Due to this reason it may not be unique. For example $ f = [x^g_1, x^h_2, x^g_1]$ can come from $ g = [x^g_1, x^h_2, x^g_3]$. Pay attention that the complete linearisation of $f$ will be $[x^g_1, x^h_2, x^g_3] + [x^g_3, x^h_2, x^g_1]$. 

The proof of the previous theorem is valid without the hypothesis of Lie algebra, that is in other contexts (associative or non-associative algebras) as long as $\dim L_g\le 1$ for each $g\in G$.. 
\end{remark}

Thus in order to determine a basis of the $\mathbb{Z}$-graded identities of $U_1$ in characteristic $p>2$ we run into the following problem. We have to give an efficient characterization of what would be a bad sequence, that is an equivalent of  Proposition \ref{good} for this case. On the other hand such an obstruction does not occur in the case of the special linear Lie algebra $sl_q(K)$ with the Pauli gradings whenever $q$ is an arbitrary prime (not necessarily equal to $p$). The reader can see in \cite[Theorem 3.13]{CDF} that the argument in its proof holds over an infinite field of characteristic $p$ different from two. 

\begin{theorem}[cf.~\cite{CDF}, Theorem 3.13]
	Let $K$ be an infinite field of characteristic $p$ different from two. Let $q$ be a prime and denote $L=sl_q(K)$ the special linear Lie algebra with the Pauli grading by the group $\mathbb{Z}_q\times \mathbb{Z}_q$. The polynomials	in (\ref{basis1}), (\ref{basis2}), (\ref{basis3}) and 
\[
[x_4,x_1,x_2,x_3]-\lambda [x_4,x_3,x_2,x_1],
\]
	for some scalar $0\ne \lambda\in K$, form a basis for the graded identities of $L$.
\end{theorem}

We proceed with the proof of Theorem \ref{mainresult2}. It suffices to consider only the $\mathbb{Z}$-graded multilinear identities for $U_1$. Although most of the results in Section \ref{identities} are valid for algebras over infinite fields of characteristic different from 2, we have to pay attention to the results of Lemmas \ref{3.6}, \ref{monosem0}, \ref{degree4} and \ref{l3}, and Propositions \ref{monident} and \ref{good}. We shall prove each one of these in our context.

\begin{lemma}\label{eq3.6}
	The graded identities \eqref{C1identp} and \eqref{C2identp} hold for $U_1$.
\end{lemma}
\begin{proof}
The proof follows from the rules of multiplication \eqref{multiwitt} and from the proof of Lemma \ref{3.6}. 
\end{proof}

\begin{lemma}
	Let $M=[x_1^{g_1},\ldots, x_n^{g_n}]$ be a multilinear monomial in $L\langle X_G \rangle$. If $M\in T_\mathbb{Z}(U_1)$ is a graded identity for $U_1$ then $M$ is a consequence of the identities \eqref{C1identp}.
\end{lemma}
\begin{proof}
	We denote by $J$ the $T_G$-ideal generated by the polynomials \eqref{C1identp}. We shall prove the claim by induction on the length $n$ of the monomial. The result is obvious for $n=1$ and for $n=2$ so we assume  $n\geq 3$. Let $M^{\prime}=[x_1^{g_1},\ldots, x_{n-1}^{g_{n-1}}]$. If $M^{\prime}\in T_\mathbb{Z}(U_1)$ then it lies in $J$ by the induction hypothesis, hence $M\in J$. We assume now $M^{\prime}\notin T_\mathbb{Z}(U_1)$. Consider $g=\| M^\prime\|$, then the result of every admissible substitution in $M^{\prime}$ is a scalar multiple of $e_g$, and for some admissible substitution the scalar is nonzero. 
	Therefore $M\in T_\mathbb{Z}(U_1)$ if and only if $[x_1^g,x_2^{g_n}]\in T_\mathbb{Z}(U_1)$. The commutator $[x_1^g,x_2^{g_n}]$ lies in $T_\mathbb{Z}(U_1)$ if and only if $g-g_n$ is a multiple of $p$, and hence $M$ lies in $J$. Thus in all cases we obtain $M\in J$, as required.
\end{proof}

The notation we introduce now is the analogue of $\mathcal{C}({\bf g})$ for an $n$-tuple ${\bf g}\in\mathbb{Z}^n$. Recall that we delete all entries of ${\bf g}$ which are equal to 0 and obtain $\mathcal{C}({\bf g})$. But when dealing with a field of positive characteristic $p$ we  can work with the integers modulo $p$. Thus we denote by $\overline{\mathcal{C}({\bf g})}$ the sequence obtained by ${\bf g}\in\mathbb{Z}^n$ after deleting the entries that are multiples of $p$. With abuse of notation we will employ the same symbol to denote the sequence formed by these non-zero elements in ${\bf g}$ but considered as belonging to the group $\mathbb{Z}_p$. Also $|\overline{\mathcal{C}({\bf g})}|$ will denote the length of the new tuple, and $|\overline{\mathcal{C}({\bf g})}|_{\bar{g}}$ will denote how many times the element $0\ne g\in \mathbb{Z}_p$ appears in $\overline{\mathcal{C}({\bf g})}$. We tacitly assume that all this notation takes place in the group $\mathbb{Z}_p$.
 
We use the same notions of \textsl{bad}, \textsl{good}, and \textsl{standard} tuples as in the case of characteristic 0.
\begin{lemma}\label{eqmonosem0}
	Let ${\bf g}\in \mathbb{Z}^{n}$ be an $n$-tuple, $n\ge 3$, containing at least two elements that are not multiples of $p$. Over an infinite field of characteristic $p>2$, the $n$-tuple ${\bf g}$ is bad if and only if $\overline{\mathcal{C}({\bf g})}$ is bad.
\end{lemma}
\begin{proof}
By Theorem \ref{identimult}, we can consider monomials in $P_n^{\bf g}$. As in Lemma \ref{monosem0} we suppose that
\begin{equation}\label{tapleorg}
{{\bf g}}=(a_1,\ldots,a_k,b_1,\ldots,b_l\underbrace{0,\ldots,0}_{n-(l+k)})
\end{equation} 
where each $\bar{a}_i\neq 0$ and $\bar{b}_i= 0$ in $\mathbb{Z}_p$. By hypothesis we have $k <n$.
	
	Let $\overline{\mathcal{C}({\bf{g}})}=(a_1,\ldots,a_k)$ be standard. Therefore, since $a_1$ is not a multiple of $p$, the monomial
\begin{equation}\label{monbom11}
[x_{1}^{a_1},x_{k+1}^{b_1},\ldots,x_{k+l}^{b_l},x_{k+l+1}^{0},\ldots,x_{n}^{0},x_2^{a_{2}},\ldots,x_k^{a_{k}}]
\end{equation}	
is not an identity of $U_1$. It follows ${\bf g}$ is also good.

Conversely, assume ${\bf g}$ is good. As before we can suppose ${\bf g}$ is standard. This implies $[x^{a_1}_1,x^{a_2}_2]$ is not in $T_{\mathbb{Z}}(U_1)$. Then renaming the variables, if necessary, and considering the tuple given by \eqref{tapleorg}, we obtain that the monomial \eqref{monbom11} is not a graded identity of $U_1$. This implies $\overline{\mathcal{C}({\bf{g}})}$ is also good. 
\end{proof}

\newpage
\begin{proposition}\label{eqgood}
Let ${\bf g}$ be a $n$-tuple in $\mathbb{Z}^{n}$ such that $|\overline{\mathcal{C}({\bf g})}|>0$.  Over an infinite field of characteristic $p>2$, the sequence $\mathcal{C}({\bf g})$ is bad if and only if 
	\begin{enumerate}
	\item[(i)] $|\overline{\mathcal{C}({\bf g})}|\neq 1$,
		\item[(ii)] $\overline{\mathcal{C}({\bf g})}$ is a sequence of elements in $g\mathbb{Z}$ modulo $p$, for some integer $g$ such that $1\leq g\leq p-1$; and,
		\item[(iii)] If some multiple of $p-g$ appears in ${\bf g}$ then
\[
\overline{\mathcal{C}({\bf g})}=(\lambda_1(p-g),\underbrace{g,\ldots,g}_{\lambda_1},\lambda_2(p-g),\underbrace{g,\ldots,g}_{\lambda_2},\ldots,\lambda_l(p-g),\underbrace{g,\ldots,g}_{\lambda_l}, g,\ldots, g),
\]
up to a permutation. Here $1\le\lambda_1\leq \lambda_2\leq \ldots\leq \lambda_l\leq p-1$. Further we require $|\overline{\mathcal{C}({\bf g})}|_{\bar{g}}\geq \lambda_1+\cdots+\lambda_l+2$. Otherwise, $\overline{\mathcal{C}({\bf g})}$ is a sequence formed only by $g$.
	\end{enumerate}
\end{proposition}
\begin{proof}
	As in Proposition~\ref{good} one sees that if (i), (ii) and (iii) hold then ${\bf g}$ is a bad sequence. Reciprocally, as a consequence of the previous lemma, we can consider tuples with nonzero elements only, that is no entry of ${\bf g}$ is a multiple of $p$, thus we assume ${\bf g}=\overline{\mathcal{C}({\bf g})}$.
	Let $j$ be the largest length of a monomial in the variables $\{x_i^{a_i}\mid i=1,\ldots,n\}$ satisfying
\[
[x_{i_1}^{a_{i_1}},\ldots, x_{i_j}^{a_{i_j}}]\notin T_\mathbb{Z}(U_1)
\]
	but
\[[x_{i_1}^{a_{i_1}},\ldots, x_{i_j}^{a_{i_j}}, x_{i_{j+1}}^{a_{i_{j+1}}}]\in T_\mathbb{Z}(U_1).
\]
An obvious induction as in Proposition~\ref{good} gives us that it suffices to consider $j=n-1$. Assume ${\bf g}$ is standard and $a_n\equiv g\pmod{p}$. We define $y=[x_{1}^{a_{1}},\ldots, x_{{n-3}}^{a_{{n-3}}}]$, $y_1=x_{{n-2}}^{a_{{n-2}}}$, $y_2=x_{{n-1}}^{a_{{n-1}}}$, and $y_3=x_{n}^{a_n}$. In this case, $\overline{\|y\|+\|y_1\|+\|y_2\|}=\bar{g}$. Notice that $[y,y_1,y_3,y_2]$ is also an identity for $U_1$. This implies one of the following must hold: either 1) $\overline{\|y\|+\|y_1\|}=\bar{g}$, or 2) $\overline{\|y\|+\|y_1\|-\|y_2\|}=(\overline{p-g})$. 

The equality in case 1) implies $\overline{\|y_2\|}=\bar{0}$ which contradicts the fact that our elements are not multiples of $p$. 

In case 2) we have $\overline{\|y\|+\|y_1\|}=\bar{0}$ and consequently $\overline{\|y_2\|}=\bar{g}$. Suppose that $\overline{\|y\|}\neq \bar{g}$, then $[y,y_2,y_3,y_1]$ is an identity for $U_1$ if and only if $\overline{\|y\|+\|y_2\|+\|y_3\|}=\overline{\|y_1\|}$. This implies $\overline{\|y\|}=-\overline{g}$ since $p\ne 2$. Therefore either $\overline{\|y\|}=\overline{\|y_2\|}=\overline{\|y_3\|}=-\overline{\|y_1\|}$, or $-\overline{\|y\|}=\overline{\|y_2\|}=\overline{\|y_3\|}=\overline{\|y_1\|}$. 
	
If $n=4$, we are done, so assume $n>4$. The sequence ${\bf\tilde{g}}=(a_1,\ldots,a_{n-3})$ is good and we can suppose it is standard. If $({a_{n-1}}-{a_{n-2}})$ is not multiple of $p$ then
$[x_{n-1}^{a_{n-1}},x_{n-2}^{a_{n-2}},x_1^{a_{1}},x_2^{a_{2}},\ldots,x_{n-3}^{a_{n-3}}]$ cannot be an identity of $U_1$. Recall that we can assume, without loss of generality, $\overline{a}_{1}\neq \overline{0}$ or $\overline{a}_2\neq \overline{0}$.

Then we repeat the argument given above. Now consider $[x_{n-1}^{a_{n-1}},x_{n-2}^{a_{n-2}}]\in T_\mathbb{Z}(U_1)$. As $n>4$, we define $\tilde{y}=[x_{1}^{a_{1}},\ldots, x_{{n-4}}^{a_{{n-4}}}]$, $\tilde{y}_1=x_{n-3}^{a_{n-3}}$, $\tilde{y}_2=x_{n-2}^{a_{n-2}}$, $\tilde{y}_3=x_{n-1}^{a_{n-1}}$, and $\tilde{y}_4=x_{n}^{a_{n}}$. Here $\overline{a}_{n}=\overline{a}_{n-1}=\overline{a}_{n-2}=\bar{g}$ and $\overline{\|\tilde{y}\|+\|\tilde{y}_1\|}=-\overline{g}$. We have to analyse two cases: 
	
Case 1.  $\overline{\|\tilde{y}\|}=\bar{g}$. Then $\overline{\|\tilde{y}_1\|}=-2{\bar{g}}$, and we assume $[x^{a_{n-3}}_{n-3},x^{a_{n-2}}_{n-2},x^{a_{n-1}}_{n-1}]$ has degree zero in $\mathbb{Z}_p$. One of the monomials
	$[x^{a_{n-3}}_{n-3},x^{a_{n-2}}_{n-2},x^{a_{n-1}}_{n-1},x^{a_{1}}_1,x^{a_{2}}_2,\ldots,x^{a_{n-4}}_{n-4},x_{n}^{a_{n}}]$ or 	$[x^{a_{n-3}}_{n-3},x^{a_{n-2}}_{n-2},x^{a_{n-1}}_{n-1},x^{a_{1}}_1,x^{a_{2}}_2,x^{a_{3}}_3,\ldots,x^{a_{n-4}}_{n-4},x_{n}^{a_{n}}]$ (or both) is not an identity of $U_1$. As ${a}_1$ and ${a}_2$ are not zero simultaneously in $\mathbb{Z}_p$ we can repeat the argument as it was done above.
	
Case 2. $\overline{\|\tilde{y}\|}\neq g$. We can assume 
	$Y=[\tilde{y}, \tilde{y}_2,\tilde{y}_3,\tilde{y}_4,\tilde{y}_1]$ is an identity for $U_1$. If $n=5$ then the only case to consider is $\overline{\|\tilde{y}\|}+3\bar{g}=\overline{\|\tilde{y}_1\|}$. This implies $\overline{\|\tilde{y}_1\|}=\bar{g}$, and the result follows. If $n>5$, we can assume $\overline{\|\tilde{y}\|}=\bar{0}$, and this implies that if $\tilde{y}=[\tilde{x},x_1]$ then $\overline{\|\tilde{x}\|}=-\overline{\|x_1\|}=\bar{a}$. But 
$[\tilde{x},\tilde{y}_1, \tilde{y}_2,\tilde{y}_3,\tilde{y}_4,x_1]$ is an identity of $U_1$. 	We assume without loss of generality that $\overline{\|x_1\|}\neq \bar{g}$. Hence 
	we repeat this argument finitely many times, and we prove the statement of (ii). Finally (iii) follows from the algorithm above.
\end{proof}

\begin{lemma}\label{eqdegree4}
	Let ${\bf g}$ be a good tuple in $\mathbb{Z}^4$. Over an infinite field of characteristic $p>2$, if $N_{\sigma} \notin T_\mathbb{Z}(U_1)$, for some $\sigma\in S_3$ then $N_\sigma	\sim [x_{1}, x_{2}, x_{3},x_{4}]$. In particular, every graded identities of degree 4 for $U_1$ lies in $I$.
\end{lemma}
\begin{proof}
The proof follows the pattern of that for Lemma~\ref{degree4}.
	We suppose, without loss of generality, that $[x_4,x_{1},x_{2},x_{3}]$ is not an identity for $U_1$. It suffices to prove that for every $\sigma \in S_3$ such that $N_\sigma\notin T_\mathbb{Z}(U_1)$, we have
\[
N_\sigma=[x_4,x_{\sigma(1)},x_{\sigma(2)},x_{\sigma(3)}]\sim [x_4,x_{1},x_{2},x_{3}].
\]
	We will analyse the following cases for $N_\sigma$. 
	
		(i) $N_\sigma=[x_4,x_{2},x_{1},x_{3}]$ is not identity for $U_1$ then our claim follows by applying the graded identity \eqref{C2identp}.

		(ii) $N_\sigma=[x_4,x_{1},x_{3},x_{2}]$ is not an identity for $U_1$. This is dealt with precisely as the case in (i).

		(iii) Let $N_\sigma=[x_4, x_3, x_1, x_2]$. 
		
		First we assume $[x_1, x_2, [x_3, x_4]]$ (consequently $[x_4, x_3, [x_1, x_2]]$) is a graded identity for $U_1$. Then $[x_4, x_3, x_2, x_1]$ cannot be an identity of $U_1$. Therefore 
		\begin{equation}\label{eq1eq}
		[x_4, x_3, x_2, x_1]\stackrel{Eq.\ \eqref{C2identp}}{\sim} [x_4, x_3, x_1, x_2].
		\end{equation} 
Observe that $[x_1, x_2, [x_3, x_4]]$ is an identity for $U_1$ if and only if either $\overline{\|x_1\|}=\overline{\|x_2\|}$, or $\overline{\|x_3\|}=\overline{\|x_4\|}$, or else $\overline{\|x_1\|+\|x_2\|}=\overline{\|x_3\|+\|x_4\|}$. Since ${\bf g}$ is standard and $[x_4, x_3, x_1, x_2]$ is not an identity for $U_1$ the first two possibilities do not occur. Thus we are left with $\overline{\|x_1\|+\|x_2\|}=\overline{\|x_3\|+\|x_4\|}$.
		 
Suppose that $\overline{\|x_3\|}=\overline{\|x_1}\|$ if and only if $\overline{\|x_4\|}=\overline{\|x_2\|}$, and we have
\[
[x_1, x_2, x_3, x_4]\equiv_I[x_1, [x_2, x_3], x_4]=[x_4, [x_2, x_3], x_1]+[x_1, x_4,[x_2,x_3]].
\]
Then $[x_1, x_4,[x_2,x_3]]$ lies in 
		$T_\mathbb{Z}(U_1)$ and therefore
\begin{equation}\label{bo1}
[x_1, x_2, x_3, x_4]\equiv_I [x_4, [x_2, x_3], x_1]\stackrel{Eq.\ \eqref{eq1eq}}{\sim}[x_4, x_3, x_1, x_2].
\end{equation}
		On the other hand
\[
[x_4, x_1, x_2, x_3]\equiv_I [x_4, [x_1, x_2], x_3]=-[x_1, x_2,x_4, x_3]\stackrel{Eq.\ \eqref{bo1}}{\sim}[x_4, x_3, x_1, x_2].
\]
		
The above argument yields that we can assume the $G$-degrees of the variables are distinct modulo $p$. Notice that
$[x_4,x_1,x_2,x_3]= [x_4,x_1,x_3,x_2]+[x_4,x_1,[x_2,x_3]]$ in every Lie algebra.	
		If $[x_4,x_1,x_3,x_2]\notin T_G(U_1)$ we are done. On the other hand we have $\overline{\|x_3\|}=\overline{\|x_1\|+\|x_4\|}$, and 
\[
[x_4,x_1,x_2,x_3]\equiv_I[x_4,x_1,[x_2,x_3]]\equiv_I[x_4,[x_1,[x_2,x_3]]]+[x_4,[x_2,x_3],x_1],
\]
that is
\[[x_4,x_1,x_2,x_3]\equiv_I[x_2,x_3,x_1,x_4]+[x_4,[x_2,x_3],x_1].
\]
If $[x_4,[x_2,x_3],x_1]$ is not an identity of $U_1$, we have the claim. Otherwise we obtain $\overline{\|x_3\|}=\overline{\|x_1\|+\|x_4\|}$, $\overline{\|x_4\|}=\overline{\|x_2\|+\|x_3\|}$, 
		\begin{equation}\label{eqmonbom}
		[x_4,x_1,x_2,x_3]\equiv_I[x_2,x_3,x_1,x_4]\equiv_I[x_1,x_3,x_2,x_4]-[x_1,x_2,x_3,x_4],
		\end{equation}		
		and
\[[x_4,x_3,x_1,x_2]=- [x_1,[x_3,x_4],x_2]\equiv_I-[x_1,x_3,x_4,x_2].
\]
		If $[x_1,x_3,x_2,x_4]$ is not an identity, we are done. Otherwise we have $\overline{\|x_1\|+\|x_3\|}=\overline{\|x_2\|}$, and 
		\begin{equation}\label{eqmonbom1}
		[x_4,x_1,x_2,x_3]\equiv_I-[x_1,x_2,x_3,x_4].
		\end{equation}
		By \eqref{eqmonbom}, the monomial $[x_3,x_2,x_1,x_4]$ is not an identity of $U_1$, and our result follows by the above equivalence, since 
		\[
		[x_4,x_3,x_2,x_1]\stackrel{Eq.\ \eqref{eqmonbom1}}{\sim}[x_3,x_2,x_1,x_4].
		\]
		
		The case where  $[x_1, x_2, [x_3, x_4]]$ (and consequently $[x_4, x_3, [x_1, x_2]]$) is not an identity for $U_1$ follows from Eq. \eqref{jacobigen}. 

		(iv) The case $[x_4, x_3, x_2, x_1]$ is analogous to the above case.

		(v) Assume $N_\sigma=[x_4,x_2,x_3,x_1]$ is not an identity of $U_1$. If neither of $[x_4,x_3,x_2,x_1]$ and $[x_4,x_2,x_1,x_3]$ is not an identity then the result follow as in (iii). 
		
		Suppose the monomial $[x_4,x_3,x_2,x_1]$ is an identity of $U_1$. This happens if and only if one of the following three conditions is met: either 1) $\overline{\|x_4\|}=\overline{\|x_3\|}$; or 2) $\overline{\|x_4\|+\|x_3\|}=\overline{\|x_2\|}$; or else 3) $\overline{\|x_4\|+\|x_3\|+\|x_2\|}=\overline{\|x_1\|}$. The latter case is impossible because $[x_4,x_2,x_3,x_1]$ is not an identity of $U_1$. Suppose that 1) holds, then
		\begin{equation}\label{eqxx1}
		[x_4,x_2,x_3,x_1]=-[x_2,x_4,x_3,x_1]\equiv_I-[x_2,x_3,x_4,x_1].
		\end{equation}		
Since $[x_4,x_2,x_1,x_3]$ is an identity of $U_1$, it follows $\overline{\|x_4\|+\|x_2\|}=\overline{\|x_1\|}$, that is $\overline{\|x_3\|+\|x_2\|}=\overline{\|x_1\|}$. This implies
\[
[x_4,x_2,x_3,x_1]\stackrel{Eq.\ \eqref{eqxx1}}{\sim}[x_2,x_3,x_4,x_1]\equiv_I[x_1,x_4,[x_2,x_3]]\stackrel{Eq.\ \eqref{C2identp}}{\sim}[x_4,x_1,x_2,x_3].
\]
		Therefore it remains to consider $\overline{\|x_4\|+\|x_3\|}=\overline{\|x_2\|}$. We have also $\overline{\|x_4\|+\|x_2\|}=\overline{\|x_1\|}$. It follows
		\begin{equation}\label{eqgg1}
		[x_4,x_1,x_2,x_3]\sim[x_1,x_2,x_4,x_3],
		\end{equation}
		and 
\[
[x_4,x_2,x_3,x_1]\sim [x_4,x_2,[x_3,x_1]]\sim [x_1,x_3,[x_4,x_2]]=[x_1,x_3,x_4,x_2]-[x_1,x_3,x_2,x_4],
\]
that is
		\begin{equation}\label{eqgg2}
		[x_4,x_2,x_3,x_1]\sim[x_1,x_3,x_4,x_2]-[x_1,x_3,x_2,x_4]
		\end{equation}
		If $[x_1,x_3,x_2,x_4]$ is not an identity of $U_1$ then the result follows, since
\[
[x_4,x_2,x_3,x_1]\stackrel{Eq.\ \eqref{eqgg2}}{\sim} [x_1,x_3,x_2,x_4]\stackrel{Eq.\ \eqref{C2identp}}{\sim} [x_1,x_2,x_3,x_4]\stackrel{Eq.\ \eqref{eqgg1}}{\sim}[x_4,x_1,x_2,x_3].
\]
Suppose $[x_1,x_3,x_2,x_4]$ is an identity, then either 
 $\overline{\|x_2\|}=\overline{\|x_1\|+\|x_3\|}$, or $\overline{\|x_1\|}=\overline{\|x_3\|}$, in addition to $\overline{\|x_4\|+\|x_2\|}=\overline{\|x_1\|}$. 
 
 If  $\overline{\|x_2\|}=\overline{\|x_1\|+\|x_3\|}$ then
\[
\overline{\|x_1\|+\|x_3\|}=\overline{\|x_4\|+\|x_3\|}\Rightarrow \overline{\|x_1\|}=\overline{\|x_4\|},
\]
and this yields that the graded monomial $[x_4,x_1,x_2,x_3]\in T_G(U_1)$ which is an absurd. 

The equality $\overline{\|x_1\|}=\overline{\|x_3\|}$ reduces to an absurd, because 
\[
\overline{\|x_2\|}=\overline{\|x_4\|+\|x_3\|}=\overline{\|x_4\|+\|x_1\|}=2\overline{\|x_4\|}+\overline{\|x_2\|}
\]
and hence $\overline{\|x_4\|}=\bar{0}$.

		The latter equality implies $\overline{\|x_1\|}=\overline{\|x_2\|}$ which contradicts the hypothesis of ${\bf g}$ being standard. 
\end{proof}

\begin{lemma}\label{eql3}
	Let $n\geq 5$ and let ${\bf g}=(a_1,\ldots,a_{n})$ be a standard $n$-tuple in $\mathbb{Z}^n$.  Over an infinite field of characteristic $p>2$, if there exists a permutation $\sigma \in S_{n-1}$ such that $N_{\sigma}$ is reducible for the distinct pairs $x_i$, $x_j$ and $x_k$, $x_l$ with $i$, $j$, $k$, $l\neq n$ then at least one of the tuples $\Delta_{i,j}({\bf g})$ and $\Delta_{k,l}({\bf g})$ is good.
\end{lemma}
\begin{proof}
The proof is similar to that of Lemma \ref{l3} and hence we omit it.
\end{proof}

By using technique similar to those of Theorem \ref{mainresult1} we obtain the proof of Theorem \ref{mainresult2}.

We this section by observing that the above results are easily adaptable to the case of the Lie algebra $W_1$. We recall that in \cite{FKK} the authors considered the base field of characteristic 0, and this was important in their proofs. Here we extend their main theorem to $W_1$ considered over an infinite field of characteristic $p\ne 2$.

\begin{theorem}
	Let $K$ be an infinite field of characteristic $p>2$. The $\mathbb{Z}$-graded identities
\[
x^c\equiv 0, \qquad [x_1^a,x_2^b] \equiv 0,
\]
	where $c<-2$, $(b-a)$ is a multiple of $p$, and
\[
\alpha[x_1^a,x_2^b, x_3^c]-\beta 	[x_1^a,x_3^c, x_2^b]\equiv 0,
\]
	where $\alpha=(c-a)(b-c-a)$, $\beta=(b-a)(c-b-a)$, form a basis for the $\mathbb{Z}$-graded identities of the Lie algebra $W_1$ over $K$.
\end{theorem}

\section{Independence of graded identities}
In this section we adapt some ideas from \cite{FKK}. Here $K$ stands for an infinite field of characteristic different from 2.

Denote by $f_{r,s}=[x^r_1, x^s_2] \in L\langle X_G\rangle$
the graded polynomials from \eqref{C1identp}. We impose the obvious restriction $r\le s$. 

\begin{lemma}
Suppose $r$, $s\in\mathbb{Z}$, $r\le s$, are such that $p$ divides $r-s$.
 The graded identity $f_{r,s}$ is not a consequence of the identities \eqref{C2identp} and the identities $f_{u,v}$ 
 where $(r,s)\neq (u,v)$.
\end{lemma}
\begin{proof}
The example given in \cite[Lemma 11]{FKK} also holds for our case. Let $r$, $s\in\mathbb{Z}$ be such that $r-s$ is a multiple of $p$, 
and let $H=UT(3, K)$ be the Lie algebra of strictly upper triangular $3\times 3$ matrices over $K$. Define the vector subspaces $H_k$ ($k\in\mathbb{Z}$) in $H$ as follows: 
\begin{enumerate}
	\item if $r\neq s$ we consider $H_k=0$ for all $k\neq r$, $s$ and $r+s$; $H_r$ is the span of $E_{12}$, $H_s$ the span of $E_{23}$ and $H_{r+s}$ the span of $E_{13}$;
	\item if $r=s=0$ then $H_0=H$ and $H_k=0$ for all $k\neq 0$; and,
	\item if $r=s\neq 0$ then $H_r$ is the linear span of $E_{12}$ and $E_{23}$, $H_{2r}$ is spanned by $E_{13}$ and $H_k=0$ for all $k\ne r$, $2r$.
\end{enumerate}
Here $E_{ij}$ is the matrix that has $1$ at the position $(i,j)$ and 0 elsewhere. It is clear that $H=\oplus_{i\in\mathbb{Z}}H_i$ is a $\mathbb{Z}$-graded Lie algebra.  Since $[E_{12},E_{23}]=E_{13}\neq0$, the graded identity $[x^r_1, x^s_2]$ is not satisfied in $H$. On the other hand, one can easily see that $H$ satisfies all graded identities \eqref{C2identp} as well as all identities $f_{u,v}$ 
when $(r,s)\neq (u,v)$. The result follows.
\end{proof}

Recall that a set $I$ of (graded) polynomials is called an independent set of (graded) identities if neither of them lies in the ideal of (graded) identities generated by the remaining ones. 

\begin{corollary}
The set of polynomials $\{f_{r,s}\mid r,s\in\mathbb{Z}, \quad r\le s, \quad r\equiv s\pmod{p}\}$ is an independent set of graded identities in $L\langle X_G\rangle$.
\end{corollary}

Now let us denote $f_{abc}=\alpha[x^a_1, x^b_2, x^c_3]-\beta[x^a_1, x^c_3, x^b_2]$
the graded identities defined in \eqref{C2identp}. Recall that $\alpha=(c-a)(b-c-a)$, $\beta=(b-a)(c-b-a)$ and that the symbol $\bar{a}$ means the residue of the division of $a$ by $p$, $0\le \bar{a}\le p-1$. If $\bar{a}=\bar{c}$ then $ \alpha=0$ but the second commutator of $f_{abc}$ vanishes due to the graded identity \eqref{C1identp}, and analogously for $\bar{b}=\bar{a}$. 
So we suppose that $\bar{b}\neq\bar{a}$ and $\bar{a}\neq\bar{c}$. Furthermore, if $\bar{b}=\overline{(a +c)}$ then $\alpha=0$ but the second commutator in $f_{abc}$ vanishes too, due to \eqref{C1identp}.  So we may suppose $\bar{b}\neq\overline{(a +c)}$ and analogously $\bar{c}=\overline{(a+b)}$. If $\bar{c}=\bar{b}$ then 
$f_{abb}$ lies in the $T_\mathbb{Z}$-ideal generated by \eqref{C1identp} since
\begin{align*}
f_{abc}&=\overline{(c-a)(b-c-a)}[x_1^a,x_2^b, x_3^c]-\overline{(b-a)(c-b-a)}[x_1^a,x_3^c, x_2^b]\\
&=-\bar{a}(\bar{b}-\bar{a})[x_1^a,x_2^b, x_3^c]+\bar{a}(\bar{b}-\bar{a})[x_1^a,x_3^c, x_2^b]\\
&=\bar{a}(\bar{b}-\bar{a})[x_1^a,[x_3^c, x_2^b]].
\end{align*}
Thus, in addition we suppose $\bar{b}\neq \bar{c}$. Similarly we treat the case $\bar{a}=\overline{b+c}$. Therefore we suppose $\bar{a}\neq\overline{b+c}$ as well.

As it was done in \cite[Section 5]{FKK}, we shall impose further restrictions on the graded identities of the type \eqref{C2identp}.

\begin{lemma}
	For all $a$, $b$, $c\in \mathbb{Z}$,
\[
f_{acb}(x_1^a,x_3^c, x_2^b)=f_{bac}(x_2^b,x_1^a, x_3^c)=-f_{abc}(x_1^a,x_2^b, x_3^c).
\]
\end{lemma}
\begin{proof}
The proof repeats that of \cite[Lemma 16]{FKK}.
\end{proof}

\begin{corollary}
		For all $a$, $b$, $c\in \mathbb{Z}$ one has 
\[
f_{bca}(x_2^b,x_3^c, x_1^a)=f_{cab}(x_3^c,x_1^a, x_2^b)=f_{abc}(x_1^a,x_2^b, x_3^c)
\]
			and
			\[
			f_{cba}(x_3^c, x_2^b,x_1^a)=-f_{abc}(x_1^a,x_2^b, x_3^c).
			\]
\end{corollary}

\begin{lemma}
	Let $K$ be an infinite field of characteristic $p>2$. Let $a$, $b$, $c\in\mathbb{Z}$ where  $a>b>c$, and $\bar{a}\neq\bar{b}$, $\bar{a}\neq \bar{c}$, $\bar{a}\neq \overline{(b+c)}$, $\bar{b}\neq \overline{(a+c)}$. Then the graded identity $f_{abc}$ is not a consequence of all identities \eqref{C1identp} and all remaining identities $f_{a^\prime b^\prime c^\prime}$ where $a^\prime>b^\prime>c^\prime$ and $(a^\prime, b^\prime, c^\prime)\neq(a,b,c)$.	
\end{lemma}
\begin{proof}
We will prove the lemma if we construct a $\mathbb{Z}$-graded Lie algebra $L^{(a,b,c)}$ such that\begin{enumerate}
	\item $L^{(a,b,c)}$ satisfies all identities \eqref{C1identp} and all identities $f_{a^\prime b^\prime c^\prime}$ where $a^\prime>b^\prime>c^\prime$, and $(a^\prime, b^\prime, c^\prime)\neq (a,b,c)$.
	\item $L^{(a,b,c)}$ does not satisfy the identity $f_{abc}$.
\end{enumerate}This yields the proof of the lemma. 

Let $L^{(a,b,c)}=UT(4,K)$ be the Lie algebra of strictly upper triangular $4\times 4$ matrices over $K$. Let $h_a=E_{12}$, $h_b=E_{23}$, $h_c=E_{34}$. Then $L^{(a,b,c)}$ is a nilpotent Lie algebra of class 3 and of dimension 6, with a $K$-basis consisting of $h_a$, $h_b$, $h_c$, $h^\prime_{a}=h_{a+b}=[h_a, h_b] =E_{13}$, $h_{b+c}=[h_b, h_c] =E_{24}$, $h^{\prime\prime}_{a}=h_{a+b+c}=[h_a, h_b, h_c] =E_{14}$. It follows  $f_{abc}$ is not a graded identity of $L^{(a,b,c)}$.

For each $i\in\mathbb{Z}$, let $H_i$ be the vector space spanned by the elements $h_i$, $h^\prime_i$, $h^{\prime\prime}_i$, if such elements exist, and $0$ otherwise. Note that $H_i=0$ if $i\notin\{a,b,c,a+b,b+c,a+b+c\}$.

{\bf Claim 1:} The graded Lie algebra $L^{(a,b,c)}$ satisfies the identities \eqref{C1identp}. In fact, if $r$ is not in $\{a, b, c, a+b,b+c,a+b+c\}$ then $H_r=0$, so
$L^{(a,b,c)}$ satisfies the identity $[x^r_1, x^s_2]$. If $r\in \{a, b, c, a+b,b+c,a+b+c\}$  then $L^{(a,b,c)}$ also satisfies the identity $[x^r_1, x^s_2]$ given in \eqref{C1identp}, because  $\bar{a}\neq\bar{b}$, $\bar{a}\neq \bar{c}$, $\bar{a}\neq \overline{(b+c)}$, $\bar{b}\neq \overline{(a+c)}$. This means $H_s=\{0\}$ or $s=r$ and $\dim H_r\leq 2$.
Also neither of the above is the case for any $r$ if $a=0$, $b=0$, $c=0$, $a+b=0$ or $b+c=0$. This implies we have to consider the following particular cases. 

Suppose that $r=a$. Since ${b}\neq{a}$, ${c}\neq {a}$, and ${b +c}\neq {a}$, we obtain $h_b$, $h_{b+c}\notin H_r$. It follows $H_r$ is contained in the span of $h_a$, $h_{a+b}$ and $h_{a+b+c}$. The latter elements commute thus the identity $[x^r_1, x^r_2]$ is satisfied in $L^{(a,b,c)}$ in this case.

Now suppose $r=b$. Similarly we get $h_a$, $h_c\notin H_r$ hence $H_r$ is contained in the span of the (commuting) elements $h_b$, $h_{b+c}$, $h_{a+b}$.                                    

Analogously, if $r =c$ then
$H_r$ is contained in the span of $h_c$, $h_{b+c}$ and $h_{a+b+c}$ that commute. Hence the identity $[x^{r}_{1}, x^{r}_{2}]$ is  satisfied in any one of these cases. 

Now assume $r\in\{a+b,b+c,a+b+c\}$. If $r=a+b$ then $H_r$ is contained in the abelian Lie subalgebra spanned by $h_{a+b}$, $h_{b+c}$ and $h_{a+b+c}$, and the graded identity \eqref{C1identp} is satisfied. The cases $r=b+c$ and $r=a+b+c$ are treated similarly.

Thus the algebra $L^{(a,b,c)}$ satisfies the identities \eqref{C1identp}.

Let $a^\prime$, $b^\prime$, $c^\prime\in\mathbb{Z}$ with $a^\prime>b^\prime>c^\prime$ and $(a^\prime, b^\prime, c^\prime)\neq(a,b,c)$. 

{\bf Claim 2:} $L^{(a,b,c)}$ satisfies the graded identities $[x^{a^\prime}_1,x^{b^\prime}_2, x^{c^\prime}_3]$ and $[x^{a^\prime}_1,x^{c^\prime}_3, x^{b^\prime}_2]$. Indeed, if $a^\prime\notin \{a, b, c, a+b,b+c,a+b+c\}$ then $H_{a^\prime}=0$, and the claim follows. In a similar way we treat the cases when $b^\prime$ or $c^\prime$ does not belong to the set $\{a, b, c, a+b,b+c,a+b+c\}$. 

Thus suppose $a^\prime$, $b^\prime$ and $c^\prime$ lie in $\{a, b, c, a+b,b+c,a+b+c\}$. As $(a^\prime, b^\prime, c^\prime)\neq(a,b,c)$, at least one of the elements $a^\prime, b^\prime, c^\prime$ does not belong to $\{a, b, c\}$. Suppose, without loss of generality, that $a^\prime\notin\{a, b, c\}$. This implies that $a^\prime \in \{a+b,b+c,a+b+c\}$, and hence, $H_i$ is contained in the vector space $V$ spanned by $E_{13}$, $E_{24}$ and $E_{14}$. It follows that $[H_{a^\prime},H_{b^\prime},H_{c^\prime}]\subseteq [V,H_{b^\prime},H_{c^\prime}]=0$, and similarly $[H_{a^\prime},H_{c^\prime},H_{b^\prime}]=0$. Claim 2 is proved.
\end{proof}

\begin{corollary}
	Over an infinite field of characteristic $p> 2$, the set of polynomials $\{f_{abc}\mid a>b>c\}$ where $a$, $b$ and $c$ satisfy 
\[
{a}\not\equiv {b}\mod p,~~{a}\not\equiv {c}\mod p, ~~{a}\not\equiv {(b+c)}\mod p, \quad {b}\not\equiv  (a+c)\pmod{p},
\]
	is an independent set of graded identities in $L\langle X_G\rangle$.
\end{corollary}

The above statements together with Theorem \ref{mainresult1} yield the following theorem.

\begin{theorem}
		Over an infinite field of characteristic $p>2$, the graded identities 
\[
[x_1^a,x_2^b] \equiv 0
\]
where $(b-a)>0$ is a multiple of $p$, and
\[
\alpha[x_1^a,x_2^b, x_3^c]-\beta 	[x_1^a,x_3^c, x_2^b]\equiv 0
\]
where the integers $a$, $b$ and $c$ with $a>b>c$ 
satisfy 
\[
{a}\not\equiv {b}\pmod{p},\ {a}\not\equiv {c}\pmod{p}, \ {a}\not\equiv {(b+c)}\pmod{p}, \ {b}\not\equiv  (a+c)\pmod{p},
\]
form a minimal basis for the $\mathbb{Z}$-graded identities of $U_1$ .
\end{theorem}

If $\mathrm{char}\ K=0$, then we consider the identities \eqref{Cbasis1} instead of \eqref{C1identp}. The next corollary is a direct consequence of the above theorem.
\begin{corollary}
Over an infinite field of characteristic different from two, the $\mathbb{Z}$-graded identities for the Lie algebra $U_1$ do not admit any finite basis.
\end{corollary}


\begin{thebibliography}{99}

\bibitem{BahKoch} Yu. Bahturin, M. Kochetov, {\it Classification of group gradings on simple Lie algebras of types $A$, $B$, $C$ and $D$}, J. Algebra {\bf 324} (2010) 2971--2989.

\bibitem{BahSehZai}	\textrm{Yu. A. Bahturin, S. K. Sehgal, M. V. Zaicev}, {\it Group gradings on associative algebras}, J. Algebra \textbf{241 (2)} (2001), 677--698.

\bibitem{BahSheZai} Y. A. Bahturin, I. P. Shestakov, M. V.  Zaicev, {\it Gradings on simple Jordan and Lie algebras},
J. Algebra {\bf 283 (2)} (2005), 849--868.

\bibitem{BahZai} Y. A. Bahturin, M. V.  Zaicev, {\it Group gradings on matrix algebras}. Dedicated to Robert V. Moody, Canad. Math. Bull. {\bf 45 (4)} (2002), 499--508.

\bibitem{DraEld} C. Draper, A. Elduque, {\it An overview of fine gradings on simple Lie superalgebras}, Note Mat. {\bf36} (2016) suppl. 1, 15--34.


\bibitem{EldKoch} A. Elduque, M. Kochetov, \textit{Gradings on simple Lie algebras}, Mathematical Surveys and Monographs, {\bf 189}. American Mathematical Society, Providence, RI; Atlantic Association for Research in the Mathematical Sciences (AARMS), Halifax, NS, 2013.

\bibitem{CDF} C. Fidelis, D. Diniz, F. L. de Souza, {\it Identities for the special linear lie algebra with the Pauli and Cartan gradings}, Israel J. Math. {\bf 241} (2021), 187-227.

\bibitem{FKK} J. A. Freitas, P. Koshlukov, A. Krasilnikov, \textit{$\mathbb{Z}$-graded identities of the Lie algebra $W_1$}, J. Algebra {\bf 427} (2015), 226--251.

\bibitem{ZG} A. Giambruno, M. Zaicev, \textit{Polynomial identities and asymptotic methods}. AMS Mathematical Surveys and Monographs Vol. {\bf 122}, Providence, R.I., 2005.

\bibitem{GM2} A. Giambruno, M. Souza, \textit{Graded polynomial identities and Specht property of the Lie algebra $sl_2$}, J. Algebra {\bf 389} (2013), 6--22.

\bibitem{GM1} A. Giambruno, M. Souza, \textit{Minimal varieties of graded Lie algebras of exponential growth and the special Lie algebra $sl_2$}, J. Pure Appl. Algebra {\bf 218} (2014), 1517--1527.

\bibitem{Guim} A. Guimar\~aes, \textit{On the support of $(\mathbb{R} +)$-gradings on the Grassmann algebra}, Commun. Algebra \textbf{49} (2021), 747-762.	

\bibitem{kac}
V. Kac, \textit{Simple  irreducible graded Lie algebras of finite growth}, Izv. Akad. Nauk USSR, Ser. Mat. \textbf{32} (l968), 1923--1967. Transl.  Math. USSR Izv. \textbf{2} (l968), 1271--1311.

\bibitem{kacbook}
V. Kac, \textit{ Infinite-Dimensional Lie Algebras}, 3rd ed., Cambridge Univ. Press, 1994.


\bibitem{K} P. Koshlukov, \textit{Graded polynomial identities for the Lie algebra $sl_2(K)$}, Internat. J. Algebra Comput. {\bf 18, no. 5} (2008),  825--836. 

\bibitem{koum}
D. Kozybaev, U. Umirbaev, \textit{Identities of the left-symmetric Witt algebras}, Internat. J. Algebra Comput. \textbf{26, No. 2} (2016), 435--450.

\bibitem{olmat}
O. Mathieu, \textit{Classification of  simple graded Lie algebras of finite growth}, Invent. Math. \textbf{108} (1992), 455--589.


\bibitem{PatZas} J. Patera, H. Zassenhaus, {\it The Pauli matrices in $n$ dimensions and finest gradings of simple Lie algebras of type $A_{n-1}$}, J. Math. Phys. {\bf 29 (3)} (1988), 665--673.

\bibitem{RZ} D. Repovs, M. Zaicev, \textit{Pauli gradings on Lie superalgebras and graded codimension growth}, Linear Algebra Appl. {\bf 520} (2017),  134--150. 

\bibitem{FP} F. Yasumura, P. Koshlukov, \textit {Asymptotics of graded codimension of upper triangular matrices}, Israel J. Math.  {\bf 223} (2018), 423--439.
\end{thebibliography}
\end{document}